\documentclass[12pt,a4paper]{article}
\pdfoutput=1
\usepackage{amsmath}
\usepackage{amsthm}
\usepackage{amssymb}
\usepackage{amsfonts}
\usepackage{mathtools}
\usepackage{color,soul}
\usepackage{verbatim}
\usepackage{enumitem}
\usepackage{float}
\usepackage{fancyhdr}
\usepackage{graphicx}
\usepackage[T1]{fontenc}

\newtheorem*{theorem*}{Theorem}
\newtheorem{theorem}{Theorem}[section]

\newtheorem*{corollary*}{Corollary}
\newtheorem{corollary}[theorem]{Corollary}

\newtheorem*{lemma*}{Lemma}
\newtheorem{lemma}[theorem]{Lemma}

\newtheorem*{claim*}{Claim}

\theoremstyle{remark}
\newtheorem*{remark*}{Remark}

\theoremstyle{definition}

\newtheorem*{definition*}{Definition}

\numberwithin{equation}{section}

\let\ker\relax
\DeclareMathOperator{\ker}{Ker}
\DeclareMathOperator{\ran}{Ran}

\title{The Method of Alternating Projections}
\author{Omer Ginat}

\begin{document}
\begin{titlepage}
    \begin{center}
        \Huge
        \vspace*{0.01cm}
        \textbf{The Method of Alternating Projections}

        \vspace*{2cm}

        \includegraphics[width=0.4\textwidth]{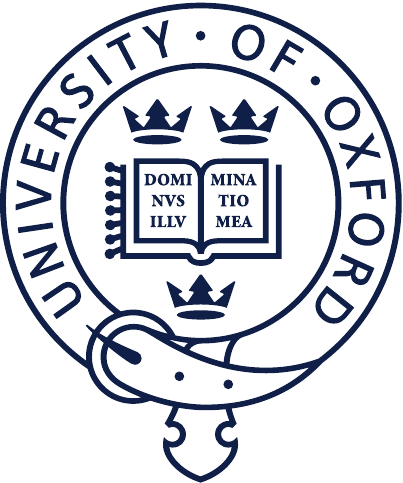}
        
        \vspace*{1.25cm}
        \LARGE
        Omer Ginat \\
        \vspace*{1cm}
        Honour School of Mathematics: Part C\\
        University of Oxford\\
        Hilary 2018 \\
        
        \vfill
        \Large
        
        Word count: 9967

    \end{center}
\end{titlepage}
\pagenumbering{roman} 
\newpage
\thispagestyle{plain}
\begin{abstract}
The method of alternating projections involves orthogonally projecting an element of a Hilbert space onto a collection of closed subspaces. It is known that the resulting sequence always converges in norm if the projections are taken periodically, or even quasiperiodically. We present proofs of such well known results, and offer an original proof for the case of two closed subspaces, known as von Neumann's theorem. Additionally, it is known that this sequence always converges with respect to the weak topology, regardless of the order projections are taken in. By focusing on projections directly, rather than the more general case of contractions considered previously in the literature, we are able to give a simpler proof of this result. We end by presenting a technical construction taken from a recent paper, of a sequence for which we do not have convergence in norm.
\end{abstract}

\newpage

\setcounter{tocdepth}{2}
\tableofcontents
\newpage
\pagenumbering{arabic}

\section{Introduction} \label{introduction}
The method of alternating projections has been widely studied in mathematics. Interesting not only for its rich theory, it also has many wide-reaching applications, for instance to the iterative solution of large linear systems, in the theory of partial differential equations, and even in image restoration; see \cite{Deu92} for a survey.

\subsection{What is the method of alternating projections?}

We begin by defining what we mean by the method of alternating projections. Let $H$ be a real or complex Hilbert space, $J\geq2$ an integer, and suppose that $M_1,\dots ,M_J$ are closed subspaces of $H$. For each $j \in \{1,\dots ,J\}$, let $P_j$ be the orthogonal projection onto the closed subspace $M_j$, and let $(j_n)_{n\geq1}$ be a sequence taking values in $\{1,\dots,J\}$. We define the sequence $(x_n)_{n\geq 0}$ by choosing an element $x_0 \in H$, and letting

\begin{equation*}
x_n = P_{j_n}x_{n-1}, \quad n\geq 1.
\end{equation*}
It is natural to ask under what conditions this sequence $(x_n)$ converges. This is often referred to as the method of alternating projections, and will be the focus of this dissertation. 

In order to motivate why we might expect $(x_n)$ to converge, it is useful to look at a simple example. Let $H = \mathbb{R}^2$, and consider the two closed subspaces 
\begin{align*}
M_1&= \{(x,y) \in \mathbb{R}^2 : x=y\}, \\
M_2 &= \{(x,y) \in \mathbb{R}^2 : y=0\}.
\end{align*}
We investigate what happens when we project $x_0 \in H$ repeatedly between $M_1$ and $M_2$ (see Figure \ref{alternating projections}).
\begin{figure}[H]
\begin{center}
\includegraphics[width=110mm]{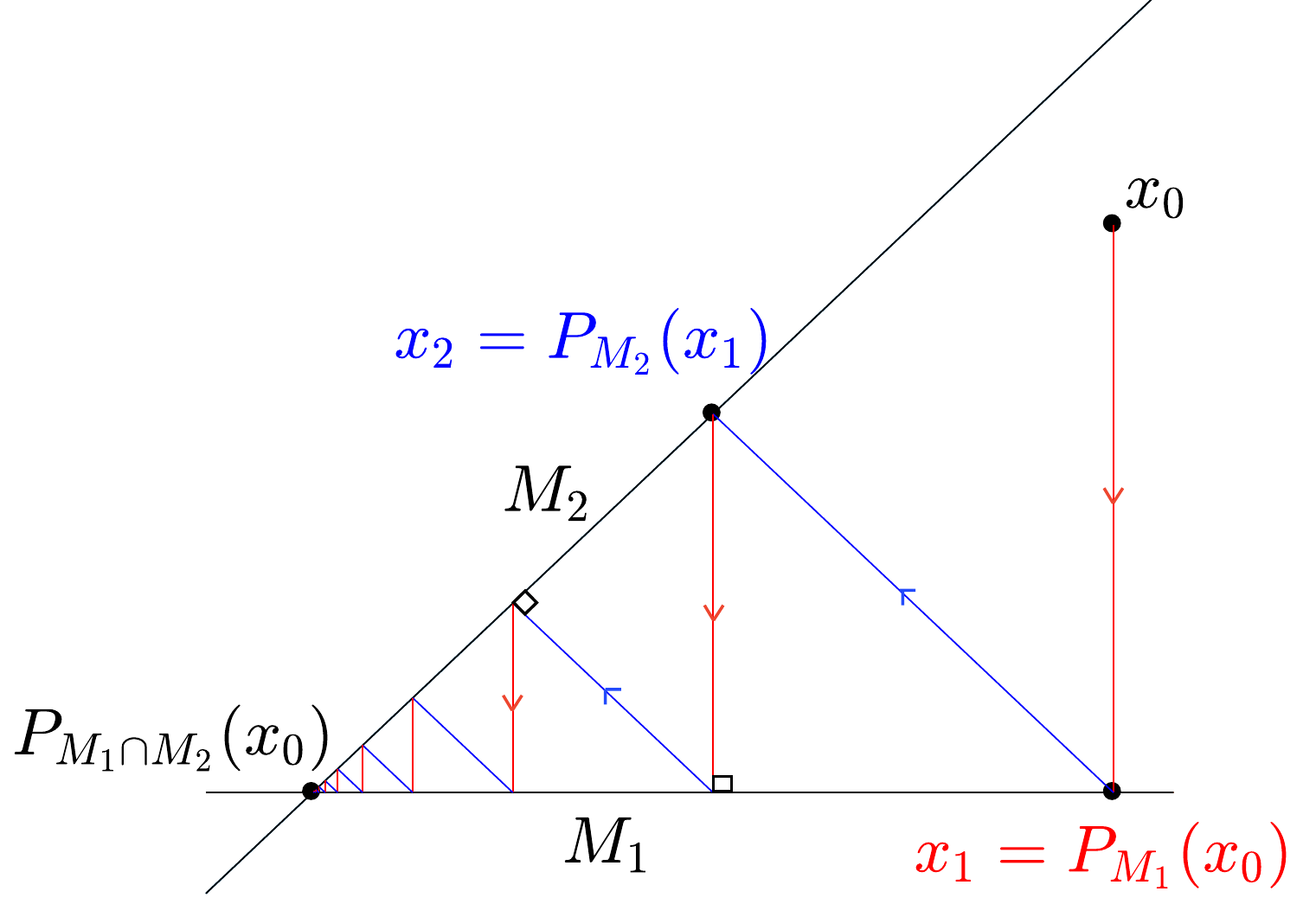}
\caption{The method of alternating projections for two subspaces of $\mathbb{R}^2$}
\label{alternating projections}
\end{center}
\end{figure}
We see that the resulting sequence converges to $(0,0)$: the projection of $x_0$ onto $M_1 \cap M_2$. More generally, we will see in Section \ref{convergence in norm} that if the sequence $(j_n)$ is taken to be periodic, then $(x_n)$ always converges in norm to the projection of $x_0$ onto $\bigcap_{j=1}^J M_j$. However, as we will observe in Section \ref{failure strong convergence}, we may find a sequence $(j_n)$ for which $(x_n)$ does not converge in norm.

In this dissertation we work through the major results relating to the convergence of $(x_n)$, sometimes offering new or more direct proofs than those in the literature today, and including important details where they have been omitted. 

\subsection{A brief history} \label{a brief history}

The first major result relating to the method of alternating projections is due to von Neumann \cite{von49}. In 1949, he proved that when we have two projections onto closed subspaces of a Hilbert space (that is, $J=2$), then $(x_n)$ converges in norm to the projection of $x_0$ onto the intersection of the two subspaces. The next significant advance happened in 1960, when Pr\'ager proved that the sequence $(x_n)$ converges in norm whenever $H$ is finite-dimensional \cite{Pra60}. Shortly after, in 1962, Halperin generalised von Neumann's theorem by proving that when the sequence $(j_n)$ is periodic, $(x_n)$ converges in norm \cite{Hal62}. \\ \\

In 1965, Amemiya and Ando proved a convergence result about products (compositions) of contractions, of which a corollary is that our sequence $(x_n)$ always converges weakly \cite{AmAn65}. This subsumes the result by Pr\'ager, since in finite-dimensional Hilbert spaces, the weak topology and norm topology coincide, and so weak convergence is equivalent to convergence in norm.

There had been no further convergence results until 1995, when Sakai improved on Halperin's theorem. He proved that when the sequence $(j_n)$ is so-called \textit{quasiperiodic}, we have convergence in norm \cite{Sak95}. Based on these positive results, it is natural to ask whether $(x_n)$ always converges in norm without restrictions on $H$ or the sequence $(j_n)$. Indeed, Amemiya and Ando posed this question in their paper \cite{AmAn65}. 

It was only in 2012 when Paszkiewicz \cite{Pas12} proved that for an infinite-dimensional Hilbert space, we may find five subspaces, a vector $x_0 \in H$, and a sequence $(j_n)$ such that $(x_n)$ does not converge in norm. In 2014, Kopeck\'a and M\"uller improved Paszkiewicz's construction from five subspaces to three \cite{KoMu14}. Indeed, this is the best we can do, since for the case of two subspaces, we are guaranteed convergence in norm by von Neumann's theorem \cite{von49}. Kopeck\'a and Paszkiewicz refined this construction in 2017. They went on to show that for any infinite-dimensional Hilbert space $H$, we may find three subspaces such that for any non-zero $x_0 \in H$, there is a sequence $(j_n)$ for which $(x_n)$ does not converge in norm \cite{KoPa17}.

\begin{figure}[H]
\begin{center}
\includegraphics[width=136mm]{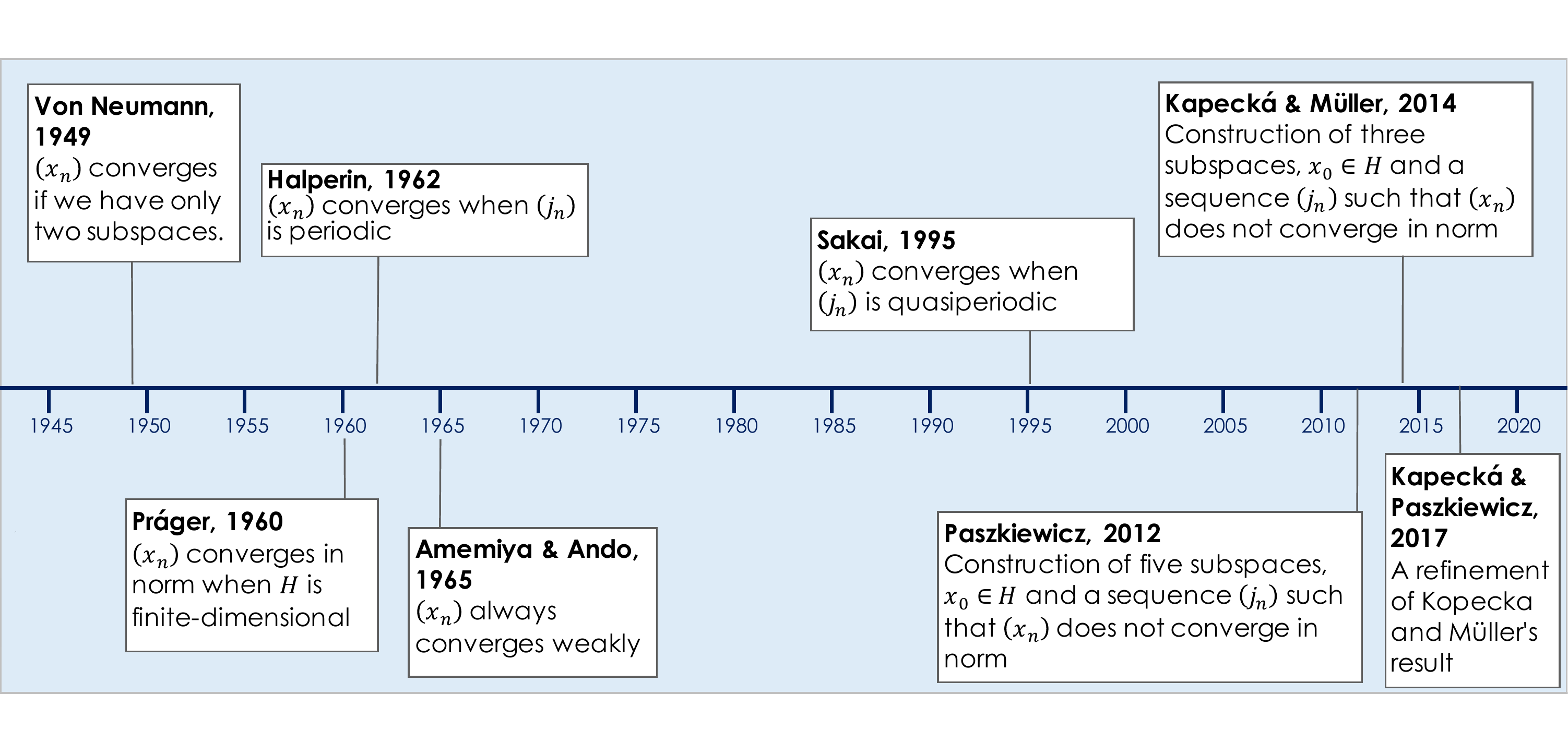}
\caption{A history of the method of alternating projections}
\end{center}
\end{figure}

\subsection{Notation} \label{notation}
Throughout this dissertation, $H$ will be a (real or complex) Hilbert space, $J\geq2$ an integer, and $M_1,\dots, M_J$ a family of closed subspaces of $H$ with intersection $M=\bigcap_{j=1}^{J} M_j$. Given a closed subspace $Y$ of $H$, we write $P_Y$ for the orthogonal projection onto $Y$, and for ease of notation, we write $P_1,\dots,P_J$ for the orthogonal projections onto $M_1, \dots ,M_J$. Throughout, $(j_n)_{n\geq1}$ will be a sequence taking values in $\{1,\dots,J\}$.  We define the sequence $(x_n)_{n\geq 0}$ by choosing a vector $x_0 \in H$, and letting
\begin{equation*}
x_n = P_{j_n}x_{n-1}, \quad n\geq1.
\end{equation*}
This will be the general setting for this dissertation. In particular, when we mention $(j_n)$ or $(x_n)$, we are referring to the sequences described above. 

We will write $B(H)$ for the space of bounded linear operators on $H$, and $B_H$ for the closed unit ball in $H$. Additionally, we write $\mathbb{F}$ for the (real or complex) scalar field of $H$. Other ad hoc notation will be introduced as needed.
\newpage

\section{Preliminaries}

We begin by recalling what it means to project orthogonally onto a closed subspace. It is a standard fact that for a closed subspace $Y$ of a Hilbert space $H$, we have $H=Y \oplus Y^\perp$. Hence each $x \in H$ can be written uniquely as $x=y+z$, where $y \in Y$ and $z \in Y^\perp$. The orthogonal projection $P_Y \colon H \to H$ onto $Y$ is given by  $P_Y(x) = y$. In fact, it is simple to see that $P_Y(x)$ is the unique closest point in $Y$ to $x$.

Before proving the important results about the method of alternating projections, it will help to introduce some elementary facts, to be referred to throughout this dissertation. The proofs are simple, but we include them to make the dissertation self-contained.

In the following lemma, we present a few elementary facts, mainly about projections.

\begin{lemma} \label{foundational}
Let $x \in H$, $Y$ be a closed subspace of $H$, and $P$ the orthogonal projection onto $Y$. Then
\begin{enumerate}[label=(\alph*)]
\item $P$ is linear, idempotent ($P^2 = P$), and self-adjoint ($P=P$*). \label{preliminary}
\item For vectors $u,v \in H$ with $u\perp v$, we have $\|u+v\|^2 = \|u\|^2 + \|v\|^2$. \label{pythagoras}
\item $\|x-Px\|^2 = \|x\|^2 - \|Px\|^2$. \label{projection equality}
\item $\|Px\| \leq \|x\|$ with equality if and only if $Px=x$. \label{projection norm}
\item $\|P\|=1$ if $Y\neq\{0\}$, and $\|P\|=0$ if $Y=\{0\}$.  \label{projection norm 1}
\item For any $x \in H$ and $y\in Y$, $\|x-Px\| \leq \|x-y\|$ with equality if and only if $Px=y$. \label{projection is onto closest point}
\item If $U$ and $V$ are closed subspaces of $H$ with $U \perp V$, then $U+V$ is closed.\label{sum projections closed}
\item If $U$ and $V$ are closed subspaces of $H$ with $U \perp V$, then $P_U+P_V = P_{U+V}$. \label{adding projections}
\end{enumerate}
\end{lemma}

\begin{proof} \ref{preliminary} For $i \in \{1,2\}$, let $x_i \in H$. Since $Y$ is a closed subspace of $H$, we have $H = Y \oplus Y^\perp$, so there are unique $y_i \in Y$ and $z_i \in Y^\perp$ such that $x_i = y_i + z_i$. For $\lambda \in \mathbb{F}$, we have \[P(x_1+\lambda x_2) = P(y_1 + \lambda y_2 + z_1 + \lambda z_2) = y_1 + \lambda y_2 = P(x_1) + \lambda P(x_2). \] Hence $P$ is linear. We also note that \[P^2(x_1) = P\big{(}P(y_1 + z_1)\big{)} = P(y_1) = y_1,\] so $P$ is idempotent. Finally, we have
\begin{equation*}
\langle Px_1,x_2 \rangle = \langle y_1,y_2 + z_2 \rangle = \langle y_1,y_2 \rangle = \langle y_1 + z_1,y_2 \rangle = \langle x_1,Px_2 \rangle.
\end{equation*}
Therefore $P$ is self-adjoint. \\ \\
\ref{pythagoras} Since $u \perp v$, we have $\langle u,v \rangle= 0 = \langle v,u \rangle$, and so \[\|u+v\|^2 = \langle u + v, u + v \rangle = \langle u,u \rangle + \langle v,v \rangle = \|u\|^2 + \|v\|^2. \]
\ref{projection equality} The result follows by applying \ref{pythagoras} with $u=Px$ and $v=x-Px$. \\ \\
\ref{projection norm} Applying \ref{projection equality}, we have that \[ \|Px\|^2 = \|x\|^2 - \|x - Px\|^2 \leq \|x\|^2,\] with equality if and only if $Px=x$. \\ \\
\ref{projection norm 1} The result follows immediately from \ref{projection norm}. \\ \\
\ref{projection is onto closest point} Since $H=Y \oplus Y^\perp$, then given $x\in H$, there are unique $\widetilde{y}\in Y$ and $\widetilde{z} \in Y^\perp$ such that $x=\widetilde{y}+\widetilde{z}$. So for any $y \in Y$, we have by \ref{pythagoras} that
\begin{equation*}
\|x-y\|^2 = \|(\widetilde{y}-y) + \widetilde{z}\|^2 = \|\widetilde{y}-y\|^2 + \|\widetilde{z}\|^2 \geq \|\widetilde{z}\|^2 = \|x - \widetilde{y}\|^2 = \|x - Px\|^2,
\end{equation*}
with equality if and only if $Px=y$. \\ \\
\ref{sum projections closed} Let $(x_n)$ be a Cauchy sequence in $U+V$. We write each $x_n$ as $u_n + v_n$, where $u_n \in U$, and $v_n \in V$. Since $U\perp V$, we have by \ref{pythagoras} that \[\|x_n - x_m\|^2 = \|(u_n-u_m) + (v_n - v_m)\|^2 =\|u_n-u_m\|^2 + \|v_n - v_m\|^2, \quad n,m \in \mathbb{N}.\] In particular, $\|u_n - u_m\| \leq \|x_n - x_m\|$ and $\|v_n - v_m\| \leq \|x_n - x_m\|$, so that $(u_n)$ and $(v_n)$ are both Cauchy sequences. Since $U$ and $V$ are closed subspaces of the Hilbert space $H$, they must be complete. Therefore $(u_n)$ and $(v_n)$ converge to some limits $u$ and $v$ respectively, and so $(x_n) = (u_n + v_n)$ converges to $u+v$. Hence $U+V$ is a complete subspace of $H$, and therefore closed. \\ \\
\ref{adding projections} By \ref{sum projections closed}, we know that $U+V$ is closed, and so $P_{U+V}$ is well defined. Since $U \perp V$, and since projections are self-adjoint, we have \[0 = \langle P_Vx,P_Uy \rangle = \langle P_UP_Vx,y \rangle = \langle x,P_VP_Uy \rangle, \quad x,y \in H.\] Hence $P_UP_V = 0 = P_VP_U$. Therefore, since $P_U$ and $P_V$ are idempotent, \[(P_U+P_V)^2 = P_U^2 + P_V^2 = P_U + P_V,\] and so $P_U + P_V$ is indeed a projection. We now note that for $x \in U$, $y\in V$, we have
\begin{equation*}
(P_U + P_V)(x+y) = P_Ux + P_Vx + P_Uy + P_Vy = x+y,
\end{equation*}
and for $z \in (U+V)^\perp$, $w \in H$, we have
\begin{equation*}
\langle (P_U+P_V)z,w \rangle = \langle z,(P_U+P_V)w \rangle = 0.
\end{equation*}
Hence $(P_U + P_V)$ is the identity on $U+V$, and zero on $(U+V)^\perp$, and so $P_U + P_V = P_{U+V}$ as claimed.
\end{proof}

The following lemma is still elementary, but the results are more specific. They will be particularly useful in proving Theorem \ref{von neumann} (von Neumann) \cite{von49} and Theorem \ref{halperin} (Halperin) \cite{Hal62}.

\begin{lemma} \label{foundational 2} Let $M_1,\dots ,M_J$ be a finite family of closed subspaces of a Hilbert space $H$, with intersection $M = \bigcap_{j=1}^J M_j $. For each $j \in \{1,\dots ,J\}$, let $P_j$ be the orthogonal projection onto the closed subspace $M_j$, and let $P_M$ be the orthogonal projection onto M. Let $T = P_J\dots P_1$. Then
\begin{enumerate}[label=(\alph*)]
\item $\bigcap_{k=1}^j \ker(I-P_k) = \ker(I-P_j\dots  P_1)$ for $j \in \{1,\dots ,J\}$. \label{useful for amemiya}
\item $Tx=x$ if and only if $x \in M$. \label{Tx=x}
\item $T^*x=x$ if and only if $x \in M$. \label{T*x=x}
\item Let $A$ be a contraction on $H$ (a bounded operator with operator norm at most $1$). Suppose that $\|A^{n+1}x - A^nx\| \to 0$ as $n\to \infty$ for every $x \in H$. Then $A^ny \to 0$ as $n\to \infty$ for every $y \in \overline{\ran(I-A)}$. \label{using kakutani}
\item If $V$ is a subspace of $H$, then $H=\overline{V} \oplus V^\perp$. \label{direct sum subspace}
\item $(\ran(I-T))^\perp = \ker(I-T^*)$. \label{ran ker equality}
\item $H=\overline{\ran(I-T)} \oplus \ker(I-T^*)$. \label{direct sum ran ker}
\end{enumerate}
\end{lemma}

\begin{proof}
\ref{useful for amemiya} If $x \in \bigcap_{k=1}^j \ker(I-P_k)$, then $P_kx=x$ for each $k \in \{1,\dots ,j\}$, and hence $P_j\dots  P_1x = x$. Conversely, if $x\in \ker(I-P_j\dots  P_1)$, then
\begin{equation*}
\|x\| = \|P_j\dots  P_1x\| \leq \|P_{j-1}\dots  P_1x\| \leq \dots  \leq \|P_1x\| \leq \|x\|.
\end{equation*}
Hence $\|P_1x\| = \|x\|$, and so by Lemma \ref{foundational}\ref{projection norm}, we have $P_1x=x$. But also $\|P_2P_1x\| = \|x\|$, so that $\|P_2x\| = \|x\|$. Lemma \ref{foundational}\ref{projection norm} then gives that $P_2x=x$. In this way, a simple induction shows that for each $k \in \{1,\dots ,j\}$, $P_kx=x$. \\ \\
\ref{Tx=x} Applying \ref{useful for amemiya} with $j=J$, we have
\begin{equation*}
Tx=x \iff x\in \ker(I-T) \iff x \in  \bigcap_{k=1}^J \ker(I-P_k) \iff x \in M.
\end{equation*}
\ref{T*x=x} An identical argument to \ref{useful for amemiya}, but with each $P_k$ replaced by $P_{j-k}$, gives \[\bigcap_{k=1}^j \ker(I-P_k) = \ker(I-P_1\dots  P_j), \quad j \in \{1,\dots ,J\}.\] We apply this with $j=J$, noting that $T^* = P_1 \dots P_J$, to get
\begin{equation*}
T^*x=x \iff x\in \ker(I-T^*) \iff x \in  \bigcap_{k=1}^J \ker(I-P_k) \iff x \in M.
\end{equation*}
\ref{using kakutani} If $x\in \ran(I-A)$, then $x=(I-A)w$ for some $w \in H$. So, by assumption,
\begin{equation*}
\|A^nx\| = \|A^n(I-A)w\| = \|A^nw - A^{n+1}w\| \to 0 \textnormal{ as } n \to \infty.
\end{equation*}
Now suppose $y \in \overline{\ran(I-A)}$. Let $\varepsilon > 0$. We can find $x \in \ran(I-A)$ such that $\|x-y\| < \varepsilon$. Then
\begin{equation*}
\|A^ny\| \leq \|A^nx\| + \|A^n(x-y)\| \leq \|A^nx\| + \|x-y\| < \|A^nx\| + \varepsilon.
\end{equation*}
Hence $\limsup_{n\to\infty} \|A^ny\| \leq \varepsilon$, and since $\varepsilon$ was arbitrary, we have $\|A^ny\| \to 0$ as $n\to \infty$. \\ \\
\ref{direct sum subspace} Since $\overline{V}$ is closed, we have $H = \overline{V} \oplus  (\overline{V})^\perp$. So we are done if we can show that $(\overline{V})^\perp = V^\perp$. Since $V\subseteq \overline{V}$, it follows that $(\overline{V})^\perp \subseteq (V)^\perp$. We now show that the other inclusion holds. 

Let $x \in (V)^\perp$, so that $\langle x,v\rangle = 0$ for all $v \in V$, and let $y \in \overline{V}$, so that there exists a sequence $y_n \in V$ converging in norm to $y$. By continuity of inner products (with one argument fixed), we have \[\langle x,y \rangle = \langle x,\lim_{n\to\infty} y_n \rangle = \lim_{n\to\infty} \langle x,y_n \rangle = \lim_{n\to\infty} 0 = 0.\] Hence $x\in (\overline{V})^\perp$, and so $(V)^\perp \subseteq (\overline{V})^\perp$. \\ \\
\ref{ran ker equality} Noting that $(I-T)^* = I-T^*$, we have 
\begin{equation*}
\begin{aligned}
x\in (\ran(I-T))^\perp &\iff \langle x,y \rangle = 0 \textnormal{ for all } y\in \ran(I-T) \\ 
&\iff \langle x,(I-T)w \rangle = 0 \textnormal{ for all } w\in H \\ 
&\iff \langle (I-T^*)x,w \rangle = 0 \textnormal{ for all } w \in H \\
&\iff (I-T^*)x = 0 \\
&\iff x\in \ker(I-T^*).
\end{aligned}
\end{equation*}
Hence $(\ran(I-T))^\perp = \ker(I-T^*)$. \\ \\
\ref{direct sum ran ker} Applying \ref{direct sum subspace} and \ref{ran ker equality}, we have
\begin{equation}
\nonumber H=\overline{\ran(I-T)} \oplus (\ran(I-T))^\perp = \overline{\ran(I-T)} \oplus \ker(I-T^*),
\end{equation}
as required.
\end{proof}

\newpage

\section{Motivation}

This dissertation focuses on proving the convergence results discussed in Section \ref{a brief history}. However, it is important to understand how these results interact with other areas of mathematics. We present three applications of the method of alternating projections beyond functional analysis. The first is a playful example in which we make use of von Neumann's theorem in the unexpected context of dividing a string into equal thirds. The other two highlight its use in finding iterative solutions to systems of linear equations and in the theory of partial differential equations.

A key theme throughout this section is that the usefulness of the method of alternating projections stems from it often being easier to compute projections onto a single closed subspace, rather than directly computing the projection onto the intersection of closed subspaces. It is also worth remarking that there are many more applications beyond the three we present in this section; see \cite{Deu92} for a survey.

\subsection{Dividing a string into equal thirds}

We begin with a charming demonstration of how von Neumann's theorem (to be proved later as Theorem \ref{von neumann}) can be applied to divide a string into equal thirds. This is due to Burkholder, and presented in the paper ``Stochastic Alternating Projections'' \cite{DiKhSa10}.

We take a string and attach two paperclips anywhere along it, calling these the `left' and `right' paperclips. We will present an iterative process, so that the positions of the paperclips will converge to one third and two thirds of the total length of the string.

At any given stage, we apply the following steps.
\begin{enumerate}[label=(\alph*)]
\item We fold over the right end of the string so that it touches the left paperclip, and slide the right paperclip until it reaches the loop. We then unfold the string.
\begin{figure}[H]
\begin{center}
\includegraphics[width=130mm]{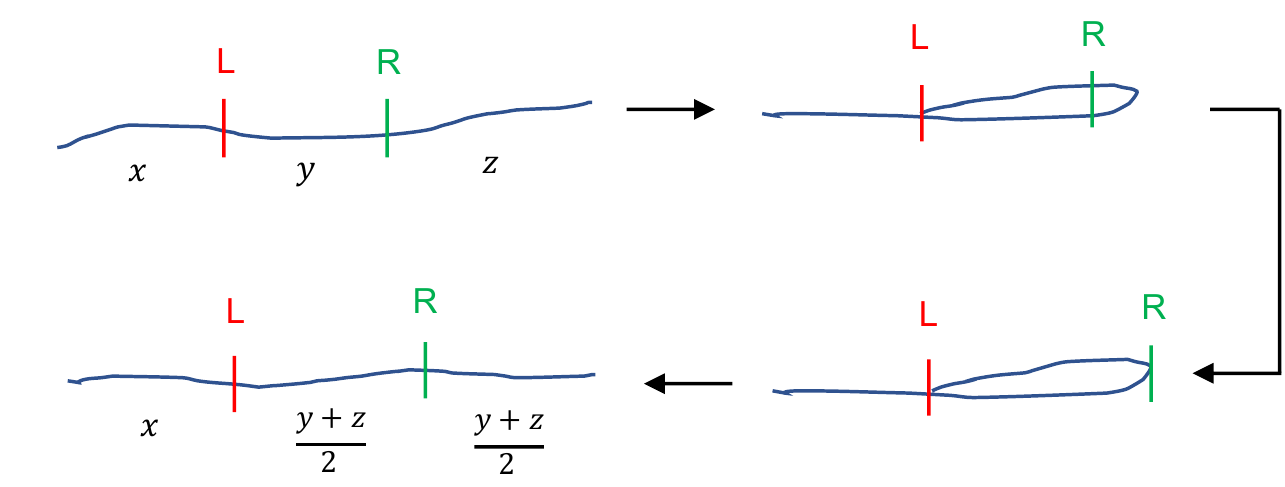}
\caption{Step (a) of the iteration}
\label{step a}
\end{center}
\end{figure}
\item This time, we fold over the left end of the string so that it touches the right paperclip, and slide the left paperclip until it reaches the loop. We then unfold the string.
\begin{figure}[H]
\begin{center}
\includegraphics[width=130mm]{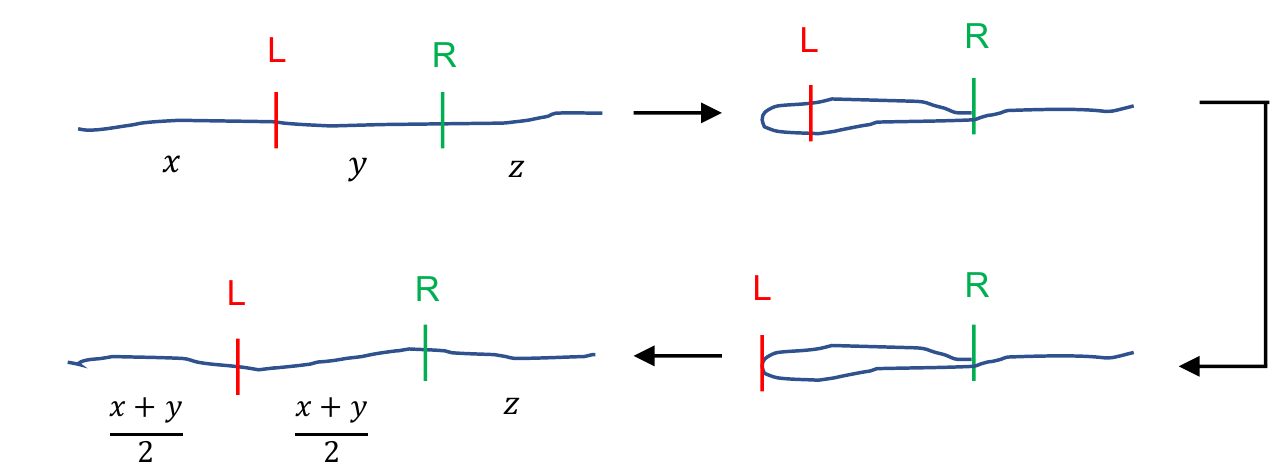}
\caption{Step (b) of the iteration}
\label{step b}
\end{center}
\end{figure}
\end{enumerate}
Applying (a) followed by (b) makes up one iteration. 
\begin{claim*}
Repeating these iterations, the positions of the paperclips converge to one third and and two thirds of the total length of the string. 
\end{claim*}

Although there are simpler ways to prove this, it is interesting to see how von Neumann's theorem may be applied in this unexpected context to give a slick proof. 

\begin{proof} Suppose the three sections of the string have lengths $x$, $y$ and $z$ respectively. Then, as shown in Figures \ref{step a} and \ref{step b}, an application of (a) leaves the three sections with lengths $x$, $(y+z)/2$, $(y+z)/2$, and an application of (b) leaves the sections with lengths $(x+y)/2$, $(x+y)/2$, and $z$. Hence, applications of (a) and (b) correspond to the projections

\[ P_1 = 
\begin{pmatrix}
1 & 0 & 0 \\
0 & 1/2 & 1/2 \\
0 & 1/2 & 1/2
\end{pmatrix}, \quad P_2 = 
\begin{pmatrix}
1/2 & 1/2 & 0 \\
1/2 & 1/2 & 0 \\
0 & 0 & 1
\end{pmatrix},\]
applied to the vector $(x,y,z)^T$.

We work in the Hilbert space $H = \mathbb{R}^3$, and let $M_1$ and $M_2$ be the closed subspaces $P_1(H)$ and $P_2(H)$ respectively. Since $P_1$ and $P_2$ are idempotent and self-adjoint, they are in fact orthogonal projections.

Let $w \in H$. For $j \in \{1,2\}$, we have  \[ w \in M_j \iff P_jw = w, \] and so \[ w \in M_1\cap M_2 \iff P_1w = P_2w = w.\] It is a simple check to see that $P_1w = P_2w = w$ if and only if $w = (a,a,a)^T$ for some $a \in \mathbb{R}$, and so, \[M_1\cap M_2 = \big{\{} (a,a,a)^T : a \in \mathbb{R} \big{\}}. \]
Hence, we have \[P_{M_1\cap M_2} \begin{pmatrix} 
x \\
y \\
z
\end{pmatrix} = \begin{pmatrix} 
\frac{x+y+z}{3} \\
\frac{x+y+z}{3} \\
\frac{x+y+z}{3}
\end{pmatrix}.  \] 
By von Neumann's theorem, which states that the limit of alternating projections onto two closed subspaces converges in norm to the projection onto the intersection of these subspaces, we have \[ \Bigg{\|}(P_1P_2)^n\begin{pmatrix} 
x \\
y \\
z
\end{pmatrix} - \begin{pmatrix} 
\frac{x+y+z}{3} \\
\frac{x+y+z}{3} \\
\frac{x+y+z}{3}
\end{pmatrix} \Bigg{\|} \to 0 \textnormal{ as $n \to \infty$}. \] Hence the positions of the paperclips converge to one third and two thirds of the total length of string, as claimed.
\end{proof}

We may also be interested in how quickly the position of the paperclips converge to one third and two thirds of the total length of the string. It turns out (after some simple calculations, which we omit here) that for a string of length $c$, the deviation of the left paperclip from $c/3$ and the right from $2c/3$, after $n$ iterations, is at most $\frac{2c}{3} \cdot 4^{-n}$ and $\frac{c}{3} \cdot 4^{1-n}$ respectively. To put this into perspective, for a string $1$ metre in length, only $3$ iterations are needed for an error of less than $1.1$ centimetres for the left paperclip.

\subsection{Solving systems of linear equations}

As mentioned in Section \ref{a brief history}, Halperin proved that the sequence obtained by periodically projecting an element of a Hilbert space orthogonally onto a collection of closed subspaces converges in norm to the projection of the element onto the intersection of these closed subspaces \cite{Hal62}. Inspired by Deutsch \cite{Deu01}, we demonstrate how Halperin's theorem (to be proved later as Theorem \ref{halperin}) can be used to find an iterative solution to a system of linear equations.

\subsubsection{The setup}

Let $H$ be a real or complex Hilbert space, $\{y_1,\dots,y_J\} \in H \setminus \{0\}$, and $\{c_1,\dots,c_J\} \in \mathbb{F}$. We want to find an element $x \in H$ satisfying the equations
\begin{equation} \label{satisfy equation}
\langle x,y_i \rangle = c_i, \quad i\in \{1,\dots, J\}. 
\end{equation}
We consider the hyperplanes
\begin{equation*}
V_i = \{y \in H \mid \langle y,y_i \rangle = c_i \}, \quad i\in \{1,\dots, J\},
\end{equation*}
and note that they are closed. We set
\begin{equation*}
V = \bigcap_{i=1}^J V_i.
\end{equation*}
Then $x$ satisfies (\ref{satisfy equation}) if and only if $x \in V$. Throughout this section, we assume a solution exists, so that $V \neq \emptyset$.

At this stage, we would like to take periodic projections of a vector $x_0 \in H$ onto the $J$ hyperplanes, and show it converges in norm to an element of $V$ (i.e. to a solution of (\ref{satisfy equation})). It seems natural to appeal to Halperin's theorem to obtain such a result. However, the hyperplanes $V_i$ may not be subspaces since they do not necessarily contain the origin.

\subsubsection{An interlude about affine spaces}

We can resolve this problem through the notion of affine spaces. There are many equivalent definitions of affine spaces, but we will use the one which is most natural in this context. 

We say that $U \subset H$ is an \textit{affine space} if $U=L+u$ for some (unique) subspace $L$ of $H$, and (any) $u \in U$. We may define a projection of $z \in H$ onto an affine space $U$ by \[P_U(z) = P_L(z - u) + u.\] This is well defined since given $u,\tilde{u} \in U$, there is some $l \in L$ such that $u = l+\tilde{u}$, and so \[P_L(u-\tilde{u}) = P_L(l) = l = u - \tilde{u}.\] Therefore by linearity of $P_L$, we have that for any $z \in H$,
\begin{equation*}
P_L(z-u) + u = P_L(z - \tilde{u}) + \tilde{u}.
\end{equation*}
For each $i \in \{1,\dots,J\}$, we let $M_i$ be the subspace given by 
\begin{equation*}
M_i = \{y \in M \mid \langle y,y_i \rangle = 0 \}. 
\end{equation*}
Then for any $i \in \{1,\dots,J\}$, we have
\begin{equation*}
V_i = M_i + v_i, \quad v_i \in V_i,
\end{equation*}
and hence each $V_i$ is an affine space. Let $M = \bigcap_{i=1}^J M_j$. It is a simple check that we have
\begin{equation*}
V = M + v, \quad v \in V.
\end{equation*}
and so $V$ is also an affine space. Therefore for any $v \in V$, we have
\begin{equation*}
\begin{aligned}
&V_i = M_i + v, \\
& V = M + v.
\end{aligned}
\end{equation*}

\subsubsection{Finding an iterative solution}

We are now in a position to be able to make use of Halperin's theorem. We begin by choosing a starting vector $x_0 \in H$, and fixing some $v \in V$.  Then for any $i,j \in \{1,\dots,J\}$, we have
\begin{equation} \label{composition of v}
\begin{aligned}
P_{V_j}P_{V_i}x_0 &= P_{V_j}(P_{M_i}(x_0-v) + v) = P_{M_j}P_{M_i}(x_0-v) + v. \\
\end{aligned}
\end{equation}
Therefore, letting $T = P_{V_J}\dots P_{V_1}$ and applying (\ref{composition of v}) repeatedly gives \[T^n x_0 = v + (P_{M_J}\dots P_{M_1})^n (x_0-v), \quad n\in \mathbb{N}.\] Hence by Halperin's Theorem,
\begin{equation*}
\|T^nx_0 - P_Vx_0\| = \|(P_{M_J}\dots P_{M_1})^n(x_0-v) - P_M(x_0 - v)\| \to 0 \textnormal{ as } n\to \infty.
\end{equation*}
In particular, since $P_Vx_0 \in V$, we see that $T^n x_0$ converges in norm to a solution of (\ref{satisfy equation}). 

By the Hilbert projection theorem, there is a unique $\tilde{v} \in V$ such that $\|x_0-\tilde{v}\|$ is minimised over $V$. We show that this unique $\tilde{v}$ is in fact $P_Vx_0$. For any $w \in V$, Lemma \ref{foundational}\ref{projection is onto closest point} gives
\begin{align*}
\|x_0-P_Vx_0\| &= \|(x_0 - w) - P_M(x_0-w)\| \leq \|(x_0-w) \|.
\end{align*}
Since $w \in V$ was arbitrary, then this unique $\tilde{v}$ is indeed $P_Vx_0$. Hence, setting $x_0=0$, we have that $T^n 0$ converges in norm to the unique minimal norm solution of (\ref{satisfy equation}). 

It is a simple check that for each $z \in H$,
\begin{equation} \label{project hyperplane}
P_{V_i}(z) = z - \frac{y_i\big{(}\langle z,y_i \rangle - c_i\big{)}}{\|y_i\|^2}, \quad i \in \{1,\dots, J\}.
\end{equation}
Thus we have a formula to easily calculate $P_{V_i}(z)$ for any $z \in H$.

A special case of particular interest is when $H = \mathbb{R}^N$ ($N \in \mathbb{N}$), where the inner product is taken to be the dot product.  Writing $x = (x_1,\dots, x_N)$ and $y_i = (a_{i1},\dots, a_{iN})$ for $i \in \{1\dots J\}$, equation (\ref{satisfy equation}) may be rewritten as
\begin{equation*}
\sum_{j=1}^N a_{ij}x_j = c_i, \quad i \in \{1,\dots, J\},
\end{equation*}
a system of linear equations. Assuming a solution exists, we have that $T^n x_0$ converges in norm to the unique solution closest in norm to our initial `guess' $x_0$. 

This is called the Kaczmarz Method, first suggested in 1937 \cite{Kac37} (see \cite{Kac93} for an English translation). Its practical value stems from our being able to easily project onto a hyperplane by using (\ref{project hyperplane}). It has a computational advantage over other known methods of solving systems of linear equations if the system is sparse. In particular, when the matrix $A=[a_{ij}]$ is sparse, the computation of $P_{V_i}(z)$ is very fast \cite{Deu92}.

\subsection{Solving PDEs on composite domains}
The final application we present is known as the Schwarz alternating method. It allows us to find an iterative solution of an elliptic partial differential equation on a region made up of two overlapping regions, in which the partial differential equation is easy to solve. We present this method for the Dirichlet problem; further examples can be found in \cite{Lio88}.

We consider the Sobolev space $H = H_0^1(\Omega)$, where the domain $\Omega = \Omega_1 \cup \Omega_2 \subset \mathbb{R}^2$ is a union of two sufficiently smooth subdomains (for example, if the domains are locally the graph of a Lipschitz continuous function). We view $H$ as a Hilbert space with inner product given by
\begin{equation*}
\langle u,v \rangle_H = \langle \nabla u, \nabla v \rangle_{L^2(\Omega)} = \int_\Omega \nabla u \cdot \overline{ \nabla v} \,dx.
\end{equation*}  
Let $\Gamma=\partial \Omega$, and for $k\in\{1,2\}$, let $\Gamma_k = \partial \Omega_k \cap \partial \Omega$ and $\gamma_k = \partial \Omega_k \setminus \partial \Omega$. 
\begin{figure}[H]
\begin{center}
\includegraphics[width=90.5mm]{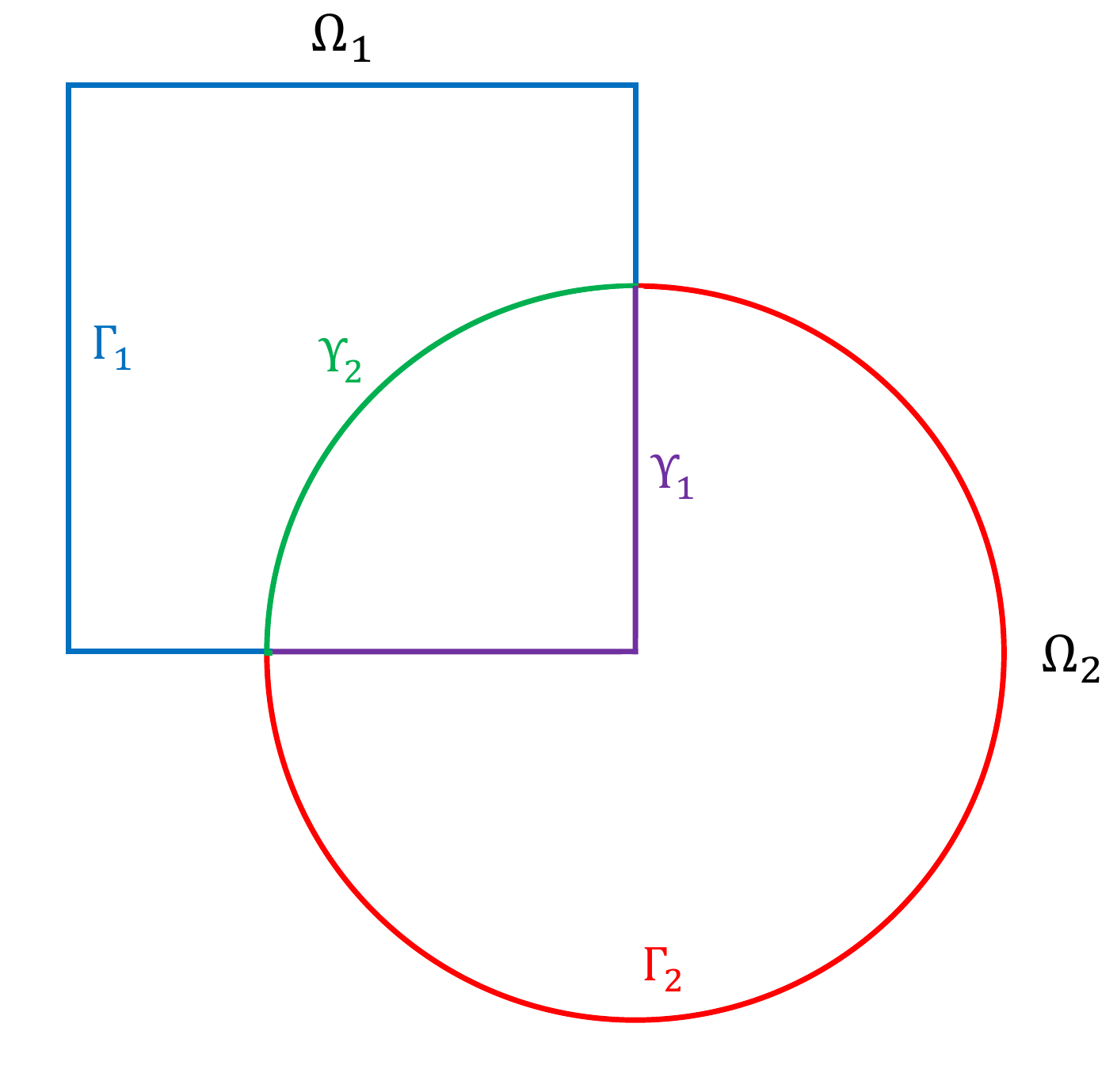}
\caption{An illustration of the domain $\Omega = \Omega_1 \cup \Omega_2$}
\label{Dirichlet}
\end{center}
\end{figure}
For $f \in L^2(\Omega)$, we would like to find a weak solution to the Dirichlet problem
\begin{equation} \label{dirichlet union}
\left\{\begin{alignedat}{2}
-\Delta &u = f && \quad \text{in}\ \Omega, \\
&u=0 && \quad  \text{on}\ \Gamma. \\
\end{alignedat}\right.
\end{equation}
Finding a weak solution means finding $u \in H$ such that we have
\begin{equation*}
\langle f,v \rangle_{L^2(\Omega)} = \langle\nabla u, \nabla v \rangle_{L^2(\Omega)}, \quad v\in H.
\end{equation*}
The motivation behind the definition of a weak solution is that for test functions $u,v \in C_c^\infty(\Omega)$, with $u$ satisfying (\ref{dirichlet union}), integration by parts gives
\begin{equation*}
\begin{aligned}
\langle f,v \rangle_{L^2(\Omega)}  = \int_\Omega f\overline{v} \,dx = \int_\Omega -\Delta u \cdot \overline{v} \,dx =  \int_\Omega \nabla u \cdot \overline{\nabla v} \,dx = \langle\nabla u, \nabla v \rangle_{L^2(\Omega)}.
\end{aligned}
\end{equation*}
Noting that $v \mapsto \langle f,v \rangle_{L^2(\Omega)}$ is a (conjugate) linear bounded functional on $H$, the Riesz representation theorem gives that there exists a unique $u \in H$ such that
\begin{equation*}
\langle f,v \rangle_{L^2(\Omega)} = \langle\nabla u, \nabla v \rangle_{L^2(\Omega)}, \quad v \in H.
\end{equation*}
That is to say, there is a unique weak solution of (\ref{dirichlet union}). In what follows, we will use von Neumann's theorem \cite{von49} to find a sequence converging in norm to the weak solution of (\ref{dirichlet union}). 

We begin by fixing $u_0 \in H$. We obtain $u_1 \in H$ by first finding a weak solution of
\begin{equation} \label{dirichlet single}
\left\{\begin{alignedat}{2}
-\Delta &u_1 = f && \quad \text{in}\ \Omega_1, \\
&u_1=0 && \quad  \text{on}\ \Gamma_1, \\
&u_1=u_0 && \quad  \text{on}\ \gamma_1, \\
\end{alignedat}\right.
\end{equation}
and then extending $u_1$ from $\Omega_1$ to $\Omega$ by letting $u_1=u_0$ on $\Omega_2 \setminus \Omega_1$. We note that finding a weak solution of (\ref{dirichlet single}) means finding $u_1 \in H^1(\Omega_1)$ with $u_1 = 0$ on $\Gamma_1$, and $u_1 = u_0$ on $\gamma_1$, such that
\begin{equation*}
\langle f,v \rangle_{L^2(\Omega_1)} = \langle\nabla u_1, \nabla v \rangle_{L^2(\Omega_1)}, \quad v \in H_0^1(\Omega_1).
\end{equation*}
The Riesz representation theorem again gives that there is a unique such $u_1$.

We then define $u_2 \in H$ by solving an analogous problem on $\Omega_2$, with $u_0$ replaced by $u_1$. Continuing in this way, we generate a sequence $(u_n)_{n\geq0}$ in $H$. We will show that $u_n$ converges in norm to $u$, the unique weak solution.

For $k\in\{1,2\}$, let $Y_k = H_0^1(\Omega_k)$, viewed as a closed subspace of $H$ after extending functions defined on $\Omega_k$ by zero to all of $\Omega$. For $k \in \{1,2\}$, we let $M_k = Y_k^\perp$, and $P_k$ be the orthogonal projection onto $M_k$. We also write $M = M_1 \cap M_2$.

Since $u$ and $u_1$ are both weak solutions of the Dirichlet problem on $\Omega_1$, we have that for every $v \in Y_1$,
\begin{equation*}
\begin{aligned}
\langle u-u_1,v \rangle_{H}&= \langle u-u_1,v \rangle_{Y_1} \\
&=\langle \nabla (u-u_1), \nabla v \rangle_{L^2(\Omega_1)} \\
&=\langle \nabla u, \nabla v \rangle_{L^2(\Omega_1)} - \langle \nabla u_1, \nabla v \rangle_{L^2(\Omega_1)} \\
&= \langle f,v \rangle_{L^2(\Omega_1)} - \langle f,v \rangle_{L^2(\Omega_1)} = 0.
\end{aligned}
\end{equation*}
Therefore $u-u_1 \in M_1$. We also note that $u_1-u_0 \in Y_1 = M_1^\perp$. Hence we have
\begin{equation*}
u-u_0 = \underbrace{(u-u_1)}_{\in M_1} + \underbrace{(u_1-u_0)}_{\in M_1^\perp},
\end{equation*}
and so $P_1(u-u_0) = u-u_1$. Similarly, we see that $P_2(u-u_1)= u-u_2$, and so on.

More generally, defining $x_n \in H$ by $x_n = u - u_{n}$ for $n \geq 1$, we have that $x_{2n+2}= P_2P_1 x_{2n}$, and so
\begin{equation*}
x_{2n} = (P_2P_1)^n x_0, \quad n\geq 1.
\end{equation*}
By von Neumann's theorem, we have
\begin{gather*}
\| x_{2n} - P_Mx_0\| \to 0, \\
\|x_{2n+1} - P_Mx_0\| = \|P_1(x_{2n} - P_Mx_0)\| \leq \| x_{2n} - P_Mx_0\| \to 0,
\end{gather*}
as $n \to \infty$, and therefore
\begin{equation*}
\|x_{n} - P_Mx_0\| \to 0 \textnormal{ as } n\to \infty.
\end{equation*}
Since $Y_1^\perp \cap Y_2^\perp = (Y_1+Y_2)^\perp$ (generally true for subspaces of a Hilbert space), and since the space $Y=Y_1 + Y_2$ can be shown to be dense in $H$, we have that $M=Y^\perp = \{0\}$. Hence $x_n \to 0$ as $n \to \infty$, and so \begin{equation*}
\|u_n - u\| \to 0 \textnormal{ as } n\to \infty.
\end{equation*}
So we have generated a sequence $u_n \in H$ converging in norm to the unique weak solution. In fact, it turns out that $u_n$ converges in norm to the weak solution exponentially fast (although this is not guaranteed if, for example, we were to switch to Neumann boundary conditions); see \cite{Lio88} for more detail.

Given Halperin's theorem, it is not surprising that we may extend this method to more than two subdomains. This extension is, again, discussed in \cite{Lio88}.

We end this section by remarking how the method of alternating projections was applied in very different ways in the examples above. While solving systems of linear equations, we used it to find an element in the intersection of closed affine subspaces. In contrast, for the Schwartz alternating method, we knew that the intersection of our subspaces was $\{0\}$, but we applied it to a sequence with terms we did not know. This highlights how versatile the method of alternating projections is, and is another reason why it has so many applications.

\newpage

\section{Convergence in norm} \label{convergence in norm}

In this section, we work through the major results that give conditions for $(x_n)$ to converge in norm, including those by von Neumann ($J=2$), Halperin (periodic projections), and Sakai (quasiperiodic projections).

\subsection{Two closed subspaces}

We will begin by proving von Neumann's theorem, that for a sequence of projections onto two closed subspaces, we are guaranteed convergence in norm \cite{von49}. 
\begin{theorem} [von Neumann] \label{von neumann}
Let $P_1,P_2$ be orthogonal projections onto the closed subspaces $M_1,M_2$ of the real or complex Hilbert space $H$, and $P_M$ the orthogonal projection onto $M=M_1\cap M_2$. Then for any $x \in H$, \[\textnormal{$\|(P_2P_1)^n x -P_Mx\| \to 0 $ as $n \to \infty$.}\]
\end{theorem}
Rather than follow von Neumann's proof, we present one which appears to not yet feature in the literature. Our proof is inspired by \cite{BaDeHu09}, where the following version of the spectral theorem is used in a similar context.
\begin{theorem*}[Spectral theorem] \label{spectral}
Let $H$ be a real or complex Hilbert space, and $T \in B(H)$ be a self-adjoint linear operator. Then there exists a measure space $(\Omega , \Sigma , \mu )$, a unitary map $U \colon H \to L^2(\Omega,\mu)$, and $m\in L^\infty(\Omega,\mu)$ with the property that $m(t) \in \mathbb{R}$ for almost all $t\in \Omega$
\begin{equation*}
UTU^{-1}f = m \cdot f, \quad f \in L^2(\Omega,\mu),
\end{equation*}
where $(m \cdot f)(t) = m(t)f(t)$ for $t \in \Omega$. Here, $\|m\|_\infty = \|T\|$.
\end{theorem*}
The idea is to consider $(P_1P_2P_1)^n$ rather than $(P_2P_1)^n$. The operator $P_1P_2P_1$ being self-adjoint allows us to apply the spectral theorem to shift the problem into some $L^2(\Omega,\mu)$, where we may make use of tools such as the dominated convergence theorem.

\begin{proof}[Proof of Theorem \ref{von neumann}] Let $T=P_1P_2P_1$. Since $P_1$ and $P_2$ are idempotent, we see that $(P_2P_1)^nx$ converges in norm to $P_Mx$ if and only if $T^nx$ does. We will prove the latter.

Since $T$ is self-adjoint, the spectral theorem gives that there exists a measure space $(\Omega , \Sigma , \mu )$, a unitary map $U: H \to L^2(\Omega,\mu)$, and $m\in L^\infty(\Omega,\mu)$ with $m(t) \in \mathbb{R}$ for almost all $t\in \Omega$, such that
\begin{equation*}
UTU^{-1}f = m \cdot f, \quad f \in L^2(\Omega,\mu).
\end{equation*}
We note that for $f \in L^2(\Omega,\mu)$, we have $m\cdot f \in L^2(\Omega,\mu)$ and also
\begin{equation*}
UT^nU^{-1}f = m^n \cdot f,\quad n\geq 0.
\end{equation*}
Let $x \in H$, and consider $f=Ux$. Noting that $UT(x) = m \cdot f$, we have 
\begin{equation*}
\begin{aligned}
m\|f\|^2 &= \langle m \cdot f,f \rangle = \langle UTx,Ux \rangle = \langle U^*UTx,x \rangle = \langle Tx,x \rangle = \langle P_1P_2P_1x,x \rangle \\ 
& = \langle P_2P_1x,P_1x \rangle = \langle P_2^2P_1x,P_1x \rangle = \langle P_2P_1x,P_2P_1x \rangle = \|P_1P_2x\|^2 \geq 0.
\end{aligned}
\end{equation*}
Hence we have $m(t) \geq 0$ for almost all $t\in \Omega$. Since $\|m\|_\infty = \|T\| \leq 1$, then $m(t) \leq 1$ for almost all $t \in \Omega$.

So for almost all $t \in \Omega$, we have $0\leq m(t) \leq 1$. Hence, defining
\begin{equation*}
\begin{alignedat}{2}
&\widetilde{\Omega} &&= \{ t \in \Omega: m(t)\leq 1\}, \\ 
&\Omega' &&= \{t \in \Omega : m(t)<1\}, \\
&\Omega^* &&= \{t \in \Omega : m(t)=1\},
\end{alignedat}
\end{equation*}
we have that $\Omega \setminus \widetilde{\Omega}$ has zero measure. Additionally, noting that $  \Omega' \cap \Omega^* = \emptyset$, $\widetilde{\Omega} =  \Omega' \cup \Omega^*$, and $1-m= 0$ on $\Omega^*$, we have

\begin{equation} \label{dct inequality}
\begin{aligned}
\Big{(}\|(m^n - m^{n+1})\cdot f\|_{L^2(\Omega,\mu)}\Big{)}^2 &= \int_{\Omega} |m^n(1-m)\cdot f|^2 \,d\mu \\
&= \int_{\widetilde{\Omega}} |m^n(1-m)\cdot f|^2 \,d\mu \\
&= \int_{\Omega'} |m^n(1-m)\cdot f|^2 \,d\mu \\
&\qquad +  \int_{\Omega^*} |m^n(1-m)\cdot f|^2 \,d\mu \\
&= \int_{\Omega'} |m^n(1-m) \cdot f|^2 \,d\mu \\
& \leq \int_{\Omega'} |m^n \cdot f|^2 \,d\mu.
\end{aligned}
\end{equation}
We note that $(m(t))^n \to 0$ as $n\to \infty$ for $t \in \Omega'$. We also note that $f$ is integrable, and so is finite almost everywhere on $\Omega$, and therefore on $\Omega'$. Hence $(m(t))^n f(t) \to 0$ as $n \to \infty$ for $t \in \Omega'$. 
Since $|m^n \cdot f| \leq |f|$ on $\Omega'$, and $f$ is integrable on $\Omega'$, we may apply the dominated convergence theorem to get
\begin{equation} \label{dct}
\lim_{n\to\infty} \int_{\Omega'} |m^n \cdot f|^2 d\mu = \int_{\Omega'} \lim_{n\to\infty} |m^n \cdot f|^2 d\mu =  \int_{\Omega'} 0 \textnormal{ } d\mu = 0.
\end{equation}
We note that $\|U^{-1}\| = 1$ (since $U$ is unitary), and that $U(T^n - T^{n+1})U^{-1}f = (m^n - m^{n+1}) \cdot f$. Applying these, along with (\ref{dct inequality}) and (\ref{dct}), gives
\begin{equation} \label{kakutani inequality}
\begin{aligned}
\|T^nx - T^{n+1}x\| &= \|U^{-1}U(T^n - T^{n+1})U^{-1}Ux\| \\
& \leq \|U^{-1}\| \cdot \|U(T^n - T^{n+1})U^{-1}f\|_{L^2(\Omega,\mu)} \\
&= \|U(T^n - T^{n+1})U^{-1}f\|_{L^2(\Omega,\mu)} \\
&= \|(m^n - m^{n+1}) \cdot f\|_{L^2(\Omega,\mu)} \\
&\leq \Big( \int_{\Omega'} |m^n \cdot f|^2 d\mu \Big)^{1/2} \textnormal{ $\to 0$ as $n \to \infty$}.
\end{aligned}
\end{equation}
We know by Lemma \ref{foundational 2}\ref{T*x=x} that $(I-T^*)x = 0 \iff x \in M$. Hence by Lemma \ref{foundational 2}\ref{direct sum ran ker}, we have
\begin{equation*}
H=\overline{\ran(I-T)} \oplus \ker(I-T^*) =\overline{\ran(I-T)} \oplus M.
\end{equation*}
So for any $x \in H$, there is a unique pair $y\in \overline{\ran(I-T)}$ and $z \in M$ such that $x=y+z$. We have by (\ref{kakutani inequality}) that $\|T^nx - T^{n+1}x\| \to 0$ as $n \to \infty$, so we can apply Lemma \ref{foundational 2}\ref{using kakutani} to see that \[\|T^n x - P_Mx\| = \|T^ny + T^n z - z\| = \|T^ny\| \to 0 \, \, \textnormal{ as } n \to \infty,\] thus concluding our proof. \end{proof}

We end this section by remarking that since projections are idempotent, any sequence of projections involving $P_1$ and $P_2$ may be reduced to one where $P_1$ and $P_2$ are alternating. Therefore Theorem \ref{von neumann} does indeed show that if $J=2$, and $(j_n)$ is any sequence, then $(x_n)$ converges in norm.

\subsection{Periodic projections}

In 1962, Halperin improved on von Neumann's theorem to show that $(x_n)$ converges in norm whenever $(j_n)$ is periodic. 

\begin{theorem} [Halperin's theorem] \label{halperin}
Let $H$ be a real or complex Hilbert space, $J \geq 2$ an integer, and $M_1,\dots,M_J$ be a collection of closed subspaces of $H$. Let $T = P_J \dots P_1$, and let $P_M$ be the orthogonal projection onto the intersection $M = \bigcap_{j=1}^J M_j$. Then for each $x\in H$, \[\textnormal{$\|T^n x - P_M x\| \to 0$ as $n\to \infty$.}\]
\end{theorem}

In this section, we follow a proof by Netyanun and Solomon \cite{NeSo06}, which makes use of Kakutani's lemma \cite{Kak40} to prove Theorem \ref{halperin} in a succinct way.

\subsubsection{Kakutani's lemma}
We begin with the core lemma of the proof, due to Kakutani \cite{Kak40}. We remark that we essentially proved Kakutani's lemma for the special case $J=2$ as part of our proof of Theorem \ref{von neumann} (von Neumann). 

\begin{lemma} [Kakutani's lemma] \label{kakutani}
Let $T=P_J\dots P_1$. For each $x \in H$, \[\| T^n x - T^{n+1} x \| \to 0 \textnormal{ as } n \to \infty.\]
\end{lemma}
\begin{proof}
Let $x \in H$. Since for each $j \in \{1,\dots, J\}$ we have $\|P_j\| \leq 1$, then $\|T^{n+1} x\| \leq \|T^n x\|$. Therefore $\|T^n x\|$ is a monotonically decreasing sequence bounded below by $0$, and so we have
\begin{equation} \label{decreasing sequence}
\|T^nx\|^2 - \|T^{n+1}x\|^2 \to 0 \textnormal{ as } n \to \infty.
\end{equation}
We now let $Q_0 = I $, and for each $ j \in \{1,\dots ,J\}$, we recursively define $Q_j = P_jQ_{j-1}$. Then
\begin{align*}
&\|T^nx - T^{n+1}x\|^2 \\
&= \|(T^n x - P_1T^n x) + (P_1T^nx - P_2P_1T^n x) + \dots  \\
&\quad \quad + (P_{J-1}\dots .P_1T^nx - P_J\dots  P_1T^n x) \|^2 \\
&= \big{\|}\sum_{j=0}^{J-1} (Q_jT^nx - Q_{j+1}T^nx)\big{\|}^2 \\
&\leq \Big{(}\sum_{j=0}^{J-1} \|Q_jT^nx - Q_{j+1}T^nx\| \Big{)}^2 && \textnormal{$\big{[}$triangle inequality$\big{]}$}\\
&\leq J \sum_{j=0}^{J-1} \|Q_jT^nx - Q_{j+1}T^nx\|^2 && \textnormal{$\Big{[}\big{(}\sum_{j=0}^{J} a_j\big{)}^2 \leq J\sum_{j=0}^{J} {a_j}^2 \Big{]}$} \\
& = J \sum_{j=0}^{J-1} (\|Q_jT^nx\|^2 - \|Q_{j+1}T^nx\|^2) && \textnormal{$\big{[}$Lemma \ref{foundational}\ref{projection equality}$\big{]}$}\\
&= J(\|Q_0T^nx\|^2 - \|Q_JT^nx\|^2) && \textnormal{$\big{[}$telescoping series$\big{]}$} \\
& = J(\|T^nx\|^2 - \|T^{n+1}x\|^2) \to 0 \textnormal{ as } n \to \infty. && \textnormal{$\big{[}Q_J=T$ and (\ref{decreasing sequence})$\big{]}$}
\end{align*}
This concludes the proof.
\end{proof}

\subsubsection{Proving Halperin's theorem}

We are now ready to prove Theorem \ref{halperin}.

\begin{proof}[Proof of Theorem \ref{halperin}]
As in the proof of Theorem \ref{von neumann}, we have by Lemma \ref{foundational 2}\ref{T*x=x} that $(I-T^*)x = 0 \iff x \in M$. Hence by Lemma \ref{foundational 2}\ref{direct sum ran ker},
\begin{equation*}
H=\overline{\ran(I-T)} \oplus \ker(I-T^*) =\overline{\ran(I-T)} \oplus M.
\end{equation*}
So for any $x \in H$, there is a unique pair $y\in \overline{\ran(I-T)}$ and $z \in M$ such that $x=y+z$. By  Lemma \ref{kakutani} (Kakutani), we have $\|T^nx - T^{n+1}x\| \to 0$ as $n \to \infty$. Therefore, we can apply Lemma \ref{foundational 2}\ref{using kakutani} to see that \[\|T^n x - P_Mx\| = \|T^ny + T^n z - z\| = \|T^ny\| \to 0 \, \, \textnormal{ as } n \to \infty,\] thus completing the proof.\end{proof}

We remark that it is in fact relatively simple to extend Theorem \ref{halperin} (Halperin), so that instead of projections, we consider contractions that are non-negative. An identical proof to that of Theorem \ref{halperin} works here; the only difference is that we do not know whether, for a non-negative contraction $A$, we have
\begin{equation} \label{inequality for kakutani}
 \|x-Ax\|^2 \leq \|x\|^2 - \|Ax\|^2, \quad x\in H.
\end{equation}
In Lemma \ref{foundational}\ref{projection equality}, we proved (\ref{inequality for kakutani}) for the special case where $A$ is a projection. It turns out it is also true more generally when $A$ is a non-negative contraction; a proof can be found in \cite{NeSo06}. We note that this extension was in fact first proved by Amemiya and Ando \cite{AmAn65}, and more recently by Bauschke, Deutsch, Hundal and Park \cite{BaDeHuPa03}.

\subsection{Quasiperiodic projections}

It turns out we can generalise Theorem \ref{halperin} (Halperin) by finding an even weaker condition than periodicity for the sequence of projections to converge in norm. In 1995, Sakai proved that we have convergence in norm if the sequence of projections is so-called \textit{quasiperiodic} \cite{Sak95}. Before defining what it means for a sequence to be quasiperiodic, or formally stating Sakai's theorem, we remind ourselves of our usual setting.

Let $H$ be a real or complex Hilbert space, $J\geq 2$ be an integer, and $M_1,\dots ,M_J$ be a family of closed subspaces of $H$ with intersection $M=\bigcap_{j=1}^J M_j$. Given a closed subspace $Y$ of $H$, we write $P_Y$ for the orthogonal projection onto $Y$, and for ease of notation, we write $P_1,\dots, P_J$ for the orthogonal projections onto the closed subspaces $M_1,\dots,M_J$ respectively. Let $(j_n)_{n\geq1}$ be a sequence taking values in $\{1,\dots,J\}$. We define the sequence $(x_n)_{n\geq0}$ by picking an arbitrary element $x_0 \in H$, and letting
\begin{equation*}
x_n = P_{j_n}x_{n-1}, \quad n\geq 1.
\end{equation*}
For ease of notation, we write
\begin{equation*}
s = (j_n)_{n\geq1}.
\end{equation*}
We now define what it means for a sequence to be quasiperiodic.
\begin{definition*}
Consider a sequence $s= (j_n)_{n\geq1}$, where each $j_n \in \{1,\dots,J\}$. We say that $s$ is \textit{quasiperiodic} if each $i \in \{1,\dots,J\}$ appears in $s$ infinitely many times, and that for each such $i$,
\begin{equation*}
I(s,i) = \sup_n\big{(}k_n(i) - k_{n-1}(i)\big{)}
\end{equation*}
is finite, where $k_0(i) = 0$, and $(k_n(i))_{n\geq1}$ is the increasing sequence of all natural numbers such that $j_{k_n(i)}=i$.

Put more simply, $s$ being quasiperiodic means that the number of entries between an element $i \in \{1,\dots,J\}$ (in the sequence $s$) and the next appearance of it, is bounded (by $I(s,i) < \infty$).
\end{definition*}

\begin{theorem}[Sakai's Theorem] \label{sakai theorem}
If $(j_n)$ is a quasiperiodic sequence, then $(x_n)$ converges in norm to the limit of the orthogonal projection of $x_0$ onto $M= \bigcap_{j=1}^J M_j$.
\end{theorem}

We follow Sakai's proof of this result \cite{Sak95}, splitting it into small subsections, each highlighting a key element or idea of the proof. 

\subsubsection{A criterion for convergence}

We begin by proving a lemma which gives us a criterion for $(x_n)$ to converge in norm.

\begin{lemma} \label{key result sakai}
Suppose there is a constant $A$ (which may depend on the sequence $s$), such that 
\begin{equation} \label{key inequality}
\|x_n - x_m\|^2 \leq A \sum_{k=m}^{n-1} \|x_{k+1} - x_k \|^2, \quad n>m\geq 1.
\end{equation}
Then the sequence $(x_n)$ converges in norm.
\end{lemma}

\begin{proof}
Since $x_{k+1}$ is the orthogonal projection of $x_k$ onto $M_{j_{k+1}}$, by Lemma \ref{foundational}\ref{projection equality} we have
\begin{equation*}
\|x_{k+1}\|^2 + \|x_{k} - x_{k+1}\|^2 = \|x_{k+1}\|^2 + \|x_{k}\|^2 - \|x_{k+1}\|^2 = \|x_k\|^2.
\end{equation*}
Hence $(\|x_k\|)$ is monotonically decreasing. Adding the equalities from $k=m$ to $k=n-1$, we obtain
\begin{equation*}
\|x_{m}\|^2 = \|x_n\|^2 + \sum_{k=m}^{n-1} \|x_{k+1} - x_{k}\|^2.
\end{equation*}
Therefore (\ref{key inequality}) is equivalent to
\begin{equation*}
\|x_n - x_m\|^2 \leq A(\|x_{m}\|^2 - \|x_n\|^2).
\end{equation*}
Since $(\|x_k\|)$ is monotonically decreasing and bounded below by $0$, we have that $\|x_k\|$ converges to some limit $c\geq0$. In particular, given $\varepsilon >0$, there exists $K \in \mathbb{N}$ such that whenever $n \geq K$,\[ 0 \leq \|x_n\| - c \leq \frac{\varepsilon}{2A}.\]
Therefore for $n,m\geq K$, we have 
\begin{equation*}
\begin{aligned}
\|x_n - x_m\|^2 &\leq A(\|x_{m}\|^2 - \|x_n\|^2) \\ 
&\leq A\left\lvert \|x_m\|^2 - c\big{|} + A\big{|}c - \|x_n\|^2 \right\lvert  \\
& < A\cdot  \frac{\varepsilon}{2A} + A\cdot  \frac{\varepsilon}{2A} \leq \varepsilon.
\end{aligned}
\end{equation*}
Hence $(x_n)$ is a Cauchy sequence, and so converges in norm (since $H$, being a Hilbert space, is complete).
\end{proof}

So under the conditions in Theorem \ref{key result sakai}, we have that the sequence $(x_n)$ converges in norm to some limit, say $x_\infty$. In particular, it also converges weakly to this limit. In the next lemma, we show that under the additional assumption that each projection appears infinitely many times in the sequence $(P_{j_n})_{n\geq 0}$, we have that $x_\infty = P_Mx_0$.

\begin{lemma} \label{key result sakai 2}
Suppose the sequence $(x_n)$ converges weakly. Suppose also that that $s = (j_n)$ takes every value in $\{1,\dots,J\}$ infinitely many times. Then the limit is the orthogonal projection of $x_0$ onto $M= \bigcap_{j=1}^J M_j$.
\end{lemma}

\begin{proof}
By assumption, $(x_n)$ converges weakly to some limit, say $x_\infty$. Each $j \in \{1,\dots,J\}$ occurs infinitely many times in $s$, and so there is some subsequence $(x_{n_k})_{k\geq1}$ such that each $x_{n_k} \in M_j$. Then for every $y \in M_j^\perp$, we have $\langle x_{n_k},y \rangle =0$, and therefore
\begin{equation*}
\langle x_\infty , y \rangle = \lim_{k\to\infty} \langle x_{n_k}, y \rangle =  \lim_{k\to\infty}0 = 0.
\end{equation*}
Hence $x_\infty \in M_j$ for each $j \in \{1,\dots,J\}$, and so $x_\infty \in M$.

To show that $x_\infty$ is the orthogonal projection of $x_0$ onto  $M$, it suffices to show that $x_0 - x_\infty \in M^\perp$, since then we would have
\begin{equation*}
x_0 = \underbrace{x_\infty}_{\in M} + \underbrace{x_0 - x_\infty}_{\in M^\perp}. 
\end{equation*}
Let $x \in M$. For every $n\geq0$, we have $(I-P_{j_{n+1}})x_n \in (M_{j_{n+1}})^\perp$ and $x \in M_{j_{n+1}}$. Therefore,
\begin{align*}
\langle x_n - x_{n+1},x \rangle &= \langle x_n - P_{j_{n+1}}x_n,x \rangle \\
&= \langle (I-P_{j_{n+1}})x_n, x \rangle \\
&=0.
\end{align*}
Adding these from $n=0$ to $n=h-1$, we have $\langle x_0 - x_h,x \rangle=0$, and so
\begin{equation*}
 \langle x_0 - x_\infty,x \rangle = \lim_{h\to \infty} \langle x_0 - x_h,x \rangle =  \lim_{h\to \infty} 0 = 0.
\end{equation*}
Hence $x_0 - x_\infty \in M^\perp$, and so $(x_n)$ converges weakly to the orthogonal projection of $x_0$ onto $M$.
\end{proof}

We now state a simple corollary of Theorems \ref{key result sakai} and \ref{key result sakai 2}.

\begin{corollary} \label{key result sakai 3}
Suppose the sequence $s$ is quasiperiodic. Suppose also that there is a constant $A$ (which may depend on the sequence $s$), such that 
\begin{equation} \label{key inequality 2}
\|x_n - x_m\|^2 \leq A \sum_{k=m}^{n-1} \|x_{k+1} - x_k \|^2, \quad n>m\geq 1.
\end{equation}
Then the sequence $(x_n)$ converges in norm to the orthogonal projection of $x_0$ onto $M$.
\end{corollary}

\begin{proof}
By Lemma \ref{key result sakai}, we immediately have that $(x_n)$ converges in norm, and so in particular converges weakly to some limit, say $x_\infty$. Since $s$ is quasiperiodic, it takes every value in \{1,\dots,J\} infinitely many times. Therefore, Lemma \ref{key result sakai 2} gives that $x_\infty = P_Mx_0$. 
\end{proof}

In particular, given a quasiperiodic sequence $s$, if we can find a constant $A$ such that (\ref{key inequality 2}) holds, then Theorem \ref{sakai theorem} (Sakai) follows immediately. The rest of this section involves finding such a constant.

\subsubsection{Useful lemmas}
We proceed to prove two simple, but useful, lemmas. These will be used in later parts of the proof of Theorem \ref{sakai theorem} (Sakai).

\begin{lemma} \label{easy inequality sakai}
Let $y_1, y_2, \dots , y_N, y_{N+1} \in H$. Then
\begin{equation*}
\|y_{N+1} - y_1\|^2 \leq N\sum_{k=1}^N \|y_{k+1} - y_k\|^2.
\end{equation*}
\end{lemma}

\begin{proof}
Applying the triangle inequality along with $\big{(}\sum_{k=1}^{N} a_k\big{)}^2 \leq N\sum_{k=1}^{N} a_k^2$, we have
\begin{align*}
\|y_{N+1} - y_1\| &= \Big{\|}\sum_{k=1}^N y_{k+1} - y_k \Big{\|}^2 \\
&\leq \Big{(}\sum_{k=1}^N \|y_{n+1} - y_n\|\Big{)}^2 \\
&\leq N\sum_{k=1}^N \|y_{n+1} - y_n\|^2. \qedhere
\end{align*}
\end{proof}

\begin{lemma} \label{small lemma}
Let $P$ be the orthogonal projection onto a closed subspace of $H$, and let $x,y \in H$. Then
\begin{enumerate}[label=(\alph*)]
\item $\|x-Py\|^2 \leq \|x-y\|^2 + \|x-Px\|^2$. \label{small lemma a}
\item $\|x-y\|^2 \leq \|x-Py\|^2 + \|x-Px\|^2 + 2\|y-Py\|^2$. \label{small lemma b}
\end{enumerate}
\end{lemma}

\begin{proof}
For (a), we note that $x-Px \perp P(x-y)$, and so Lemma \ref{foundational}\ref{pythagoras} gives
\begin{align*}
\|x-Py\|^2 &= \|x-Px+P(x-y)\|^2 \\
& = \|x-Px\|^2 + \|P(x-y)\|^2 \\
&\leq \|x-Px\|^2 + \|x-y\|^2.
\end{align*}
For (b), we begin by noting that since $Px,Py \perp y-Py$, we have
\begin{align*}
\langle x-Py,y-Py \rangle &= \langle x,y-Py \rangle = \langle x-Px,y-Py \rangle.
\end{align*}
Therefore, appealing to the Cauchy-Schwartz inequality and noting that $2ab \leq a^2+b^2$, we have
\begin{align*}
\|x-y\|^2 &= \|x-Py-(y-Py))\|^2 \\
&\leq \|x-Py\|^2 + \|y-Py\|^2 + 2|\langle x-Py,y-Py \rangle| \\
&= \|x-Py\|^2 + \|y-Py\|^2 + 2|\langle x-Px,y-Py \rangle| \\
&\leq \|x-Py\|^2 + \|y-Py\|^2 + 2\|x-Px\|\|y-Py\| \\
&\leq \|x-Py\|^2 + \|y-Py\|^2 + \|x-Px\|^2 + \|y-Py\|^2 \\
&\leq \|x-Py\|^2 + \|x-Px\|^2 + 2\|y-Py\|^2. \qedhere
\end{align*}
\end{proof}

\subsubsection{Two statements implying Sakai's theorem}
There are two steps left in the proof. We will first find two statements from which Theorem \ref{sakai theorem} (Sakai) follows, and then we will show these statements are true. In this subsection, we do the former.

We begin by defining $I = I(s) = \sup_{1\leq j\leq J} I(s,j)$, and \[S_l = \sum_{k=l}^{l+I-2} \|x_{k+1} - x_k\|^2.\]
By Lemma \ref{key result sakai 3}, to prove Theorem \ref{sakai theorem} (Sakai), it suffices to show that
\begin{equation} \label{want to prove sakai}
\|x_n - x_m\|^2 \leq \big{(}(I(s)-1)(I(s)-2)+3\big{)} \sum_{k=m}^{n-1} \|x_{k+1} - x_k\|^2, \quad n>m\geq1. \, \, \,
\end{equation}
Let $n \geq m \geq 1$. Suppose first that \[ \textnormal{(a)} \quad n-m \leq 2I-3. \] Then by Lemma \ref{easy inequality sakai} (with $N=n-m$, $y_1 = x_m$, $y_2 = x_{m+1}$,$\dots$, $y_N=x_{n-1}$, $y_{N+1} = x_n$), we have

\begin{equation} \label{almost there}
\|x_n-x_m\|^2 \leq (n-m)\sum_{k=m}^{n-1} \|x_{k+1} - x_k\|^2.
\end{equation}
Since $n-m \leq 2I-3 \leq (I-1)(I-2) + 3$, we see that (\ref{want to prove sakai}) holds.

We may therefore assume that \[\textnormal{(b)} \quad n-m \geq 2I-2. \] We now note that in order to show (\ref{want to prove sakai}) holds, it is sufficient to prove the following statements.

\noindent (i) If $S_{n-I+1} \leq S_m$, then
\begin{equation*}
\|x_n - x_m\|^2 \leq \|x_n - x_{m+I-1}\|^2 + \big{(}(I-1)(I-2) + 3\big{)}S_m.
\end{equation*}
(ii) If $S_m < S_{n-I+1}$, then
\begin{equation*}
\|x_n - x_m\|^2 \leq \|x_{n-I+1} - x_m\|^2 + \big{(}(I-1)(I-2) + 3\big{)}S_{n-I+1}.
\end{equation*}
Indeed, if we have (i) and (ii), then we can apply them repeatedly (whichever of the two we are able to apply at each step), until we are in case (a). For example, suppose we have just applied (i) to $\|x_n - x_m\|$. Then either $n-(m+I-1) \geq 2I-2$, so that we can apply one of (i) or (ii) to $\|x_n - x_{m+I-1}\|^2$, or $n-(m+I-1) \leq 2I-2$, so that we are in case (a) and we get, as in (\ref{almost there}),
\begin{equation*}
\|x_n - x_{m+I-1}\|^2 \leq (n-(m+I-1))\sum_{k=m+I-1}^{n-1} \|x_{k+1} - x_k\|^2. 
\end{equation*}
After repeated applications of (i) or (ii), and once we are in case (a) so that we have a similar inequality to (\ref{almost there}), we obtain (\ref{want to prove sakai}), and we are done.

\subsubsection{Proving Sakai's theorem}

In the last section, we showed that in order to prove Theorem \ref{sakai theorem} (Sakai), it suffices to prove that for $n-m \leq 2I-2$, the following two statements hold.

\noindent (i) If $S_{n-I+1} \leq S_m$, then
\begin{equation*}
\|x_n - x_m\|^2 \leq \|x_n - x_{m+I-1}\|^2 + \big{(}(I-1)(I-2) + 3\big{)}S_m.
\end{equation*}
(ii) If $S_m < S_{n-I+1}$, then
\begin{equation*}
\|x_n - x_m\|^2 \leq \|x_{n-I+1} - x_m\|^2 + \big{(}(I-1)(I-2) + 3\big{)}S_{n-I+1}.
\end{equation*}

\begin{proof}[Proof of (i)] For $k \in \{m,m+1, \dots, m+I-2\}$, applying Lemma \ref{small lemma}\ref{small lemma b} with $x=x_n,y=x_k$, $P=P_{j_{k+1}}$, and setting $p_{j_{k+1}}=P_{j_{k+1}}x_n$, we have
\begin{equation*}
\|x_n-x_k\|^2 \leq \|x_n-x_{k+1}\|^2 + \|x_n - p_{j_{k+1}}\|^2 + 2\|x_{k+1} - x_k\|^2.
\end{equation*}
Applying this inequality one by one to each of $\|x_n-x_m\|^2$, $\|x_n-x_{m+1}\|^2,\dots,$ $\|x_n-x_{m+I-2}\|^2$, we obtain
\begin{equation} \label{intermediate inequality}
\begin{aligned}
\|x_n-x_m\|^2 &\leq \|x_n-x_{m+1}\|^2 + \|x_n - p_{j_{m+1}}\|^2 + 2\|x_{m+1} - x_m\|^2 \\ 
&\leq \dots \leq \|x_n-x_{m+I-1}\|^2 + \sum_{k=m}^{m+I-2} \|x_n - p_{j_{k+1}}\|^2 + 2S_m.
\end{aligned}
\end{equation}
The set $\{ x_{n-I+1},x_{n-I+2},\dots , x_n\}$ consists of $I$ consecutive elements of the sequence $(x_n)$. So by definition of $I$, at least one of these elements, say $x_h$, belongs to $M_{j_{k+1}}$. We choose the largest such number $h$, and denote it by $h(j_{k+1})$.

Since $p_{j_{k+1}}=P_{j_{k+1}}x_n$ is the projection of $x_n$ onto $M_{j_{k+1}}$, and $x_{h(j_{k+1})} \in M_{j_{k+1}}$, Lemma \ref{foundational}\ref{projection is onto closest point} gives
\begin{equation} \label{closest point inequality}
\|x_n - p_{j_{k+1}}\|^2 \leq \|x_n - x_{h(j_{k+1})}\|^2.
\end{equation}
Applying Lemma \ref{easy inequality sakai}, we obtain
\begin{equation} \label{intermediate inequality 2}
\begin{aligned}
\|x_n - x_{h(j_{k+1})}\|^2 & \leq (n-h(j_{k+1})) \sum_{k=h(j_{k+1})}^{n-1} \|x_{k+1} - x_k\|^2 \\
& \leq (n-h(j_{k+1})) S_{n-I+1}.
\end{aligned}
\end{equation}
Since $k \in \{m,\dots,m+I-2\}$, $k$ ranges over $I-1$ consecutive numbers. Therefore, there is some number $a$ in this range such that $M_{j_{a+1}}$ is equal to one of $M_{j_{n-1}}$ or $M_{j_{n}}$. Rephrasing this, there is some $a \in \{m,\dots,m+I-2\}$ for which $n-h(j_{a+1})$ is equal to $0$ or $1$. Since $ 0 \leq n-h(j_{k+1}) \leq I-1$,  we have
\begin{equation} \label{intermediate inequality 3}
\begin{aligned}
&\sum_{k=m}^{m+I-2} n-h(j_{k+1}) \\
&\leq \sum_{k=m}^{a-1} n-h(j_{k+1}) + \Big{(}n-h(j_{a+1})\Big{)} + \sum_{k=a+1}^{m+I-2} n-h(j_{k+1}) \\
& \leq \Big{(}\sum_{k=m}^{a-1} I-1 \Big{)} + 1 + \Big{(}\sum_{k=a+1}^{m+I-2} I-1 \Big{)} \\
& \leq (I-1)(I-2) + 1. 
\end{aligned}
\end{equation}
Hence applying (\ref{closest point inequality}), (\ref{intermediate inequality 2}) and (\ref{intermediate inequality 3}) in that order, and recalling that $S_{n-I+1} \leq S_m$ (by assumption), we have

\begin{equation} \label{intermediate inequality 4}
\begin{aligned}
\sum_{k=m}^{m+I-2} \|x_n - p_{j_{k+1}}\|^2 + 2S_m & \leq \sum_{k=m}^{m+I-2} \|x_n - x_{h(j_{k+1})}\|^2 + 2S_m \\
& \leq \sum_{k=m}^{m+I-2} (n-h(j_{k+1})) S_{n-I+1} + 2S_m  \\
& \leq \big{(}(I-1)(I-2) + 1\big{)}S_{n-I+1} + 2S_m \\
& \leq \big{(}(I-1)(I-2) + 3\big{)}S_m.
\end{aligned}
\end{equation}
So finally, by (\ref{intermediate inequality}) and (\ref{intermediate inequality 4}), we have
\begin{equation*}
\begin{aligned}
\|x_n-x_m\|^2 & \leq \|x_n-x_{m+I-1}\|^2 + \sum_{k=m}^{m+I-2} \|x_n - p_{j_{k+1}}\|^2 + 2S_m \\
& \leq \|x_n-x_{m+I-1}\|^2 + \big{(}(I-1)(I-2) + 3\big{)}S_m,
\end{aligned}
\end{equation*}
and so (i) is proved.
\end{proof} 

\begin{proof}[Proof of (ii)]
For each $k \in \{n-I+1,\dots, n-1\}$, applying Lemma \ref{small lemma}\ref{small lemma a} with $x=x_m$, $y=x_k$, $P=P_{j_{k+1}}$, and setting $p_{j_{k+1}}=P_{j_{k+1}}x_m$, we have
\begin{equation*}
\|x_m - x_n\|^2 \leq \|x_m - x_k\|^2 + \|x_m - p_{j_{k+1}} \|^2. 
\end{equation*}
Applying these inequalities repeatedly (as in the proof of (i)), we obtain
\begin{equation} \label{intermediate inequality b1}
\|x_m - x_n\|^2 \leq \|x_m - x_{n-I+1}\|^2 + \sum_{k=n-I+1}^{n-1} \|x_m - p_{j_{k+1}}\|^2.
\end{equation}
An argument similar to  (\ref{intermediate inequality 4}) shows that we have
\begin{equation} \label{intermediate inequality b2}
\sum_{k=n-I+1}^{n-1} \|x_m - p_{j_{k+1}}\|^2 \leq \{(I-1)(I-2) + 1\}S_m.
\end{equation}
Combining (\ref{intermediate inequality b1}) and (\ref{intermediate inequality b2}), and recalling that $S_m<S_{n-I+1}$, we obtain (ii) and so the proof is complete.
\end{proof}

\subsubsection{Concluding remarks} \label{concluding remarks}

Sakai's paper ends by posing several open questions about the convergence of sequences of projections. He also mentions a few simple results. In this subsection we briefly discuss some of the questions and ideas raised by Sakai.

The first question he poses is the following. For an arbitrary sequence $s$, does (\ref{key inequality}) always hold with $A=J-1$?

We remark that it appears this question has not yet been addressed in the literature. Perhaps this is because it can be easily resolved given a result by Kopeck\'a and Paszkiewicz \cite{KoPa17} (stated later as Theorem \ref{big result kopecka}). We resolve Sakai's question in Corollary \ref{Sakai open question}, where we find a sequence $s$ for which (\ref{key inequality}) does not hold for any constant $A$.

Another interesting question posed by Sakai is whether we still have convergence in norm for the case that $J= \infty$ and $(j_n)$ is quasiperiodic. It is worth noting that quasiperiodic sequences covering every integer do exist. We offer the following example: 
\begin{equation*}
1,2,1,3,1,2,1,4,1,2,1,3,1,2,1,5,\dots 
\end{equation*}
That is, the sequence which has $1$ every $2$\textsuperscript{nd} number, $2$ every $4\textsuperscript{th}$ number, \dots, $n$ every $2^n$-th number, etc. 

More generally, for the case $J=\infty$, a quasiperiodic sequence always has $I = \sup_{j\in \mathbb{N}}I(s,j) = \infty$, and so the argument in the proof of Theorem \ref{sakai theorem} (Sakai) does not extend to this case. However, even when $J= \infty$, we can still show convergence for special cases. 

For example, suppose the sequence of closed subspaces $(M_j)_{j\geq1}$ is monotonically decreasing (i.e. $M_1\supseteq M_2 \supseteq M_3 \supseteq \dots $). Then we have
\begin{equation} \label{commuting projections}
P_bP_a = P_aP_b = P_b, \quad b\geq a \geq1.
\end{equation}
Consider the sequence $s$ given by $j_n = n$ for every $n \in \mathbb{N}$. Applying (\ref{commuting projections}) and Lemma \ref{foundational}\ref{projection equality} for the first equality gives that for $n\geq m \geq 1$,
\begin{equation*}
\begin{aligned}
\|x_n - x_m\|^2 &= \|x_m\|^2 - \|x_n\|^2 \\
&= \|x_m - x_{m+1} + x_{m+1} - x_{m+2} + \dots  - x_n + x_n\|^2 - \|x_n\|^2 \\
& \leq \|x_m - x_{m+1}\|^2 + \dots  + \|x_{n-1} - x_n\|^2 + \|x_n\|^2 - \|x_n\|^2 \\
& = \sum_{k=m}^{n-1} \|x_{k+1} - x_k \|^2.
\end{aligned}
\end{equation*}
Hence (\ref{key inequality}) holds with $A=1$, and so $(x_n)$ convergences in norm.

Finally, we note that in his paper, Sakai observed that if at least one of the $J$ subspaces is finite-dimensional, then for any sequence $(j_n)$, we have that $(x_n)$ converges in norm. To prove this, we will need to make use of a result by Amemiya and Ando, to be stated later as Theorem \ref{amemiya}. We therefore defer this result (and its proof) to Section \ref{amemiya section}, where we state it as Lemma \ref{Sakai mistake}.

\newpage

\section{Weak convergence} \label{amemiya section}

All of our results so far have had some restriction on the sequence of projections, or on the Hilbert space $H$. It is natural to ask what happens if we do not have such restrictions. In 1965, Amemiya and Ando proved that for any sequence of projections, we always have weak convergence \cite{AmAn65}.

\begin{theorem} \label{amemiya}
Let $H$ be a real or complex Hilbert space, $J \geq 2$ an integer, and $M_1,\dots ,M_J$ a family of closed subspaces in $H$. For each $j \in \{1,\dots ,J\}$, let $P_j$ be the orthogonal projection onto the closed subspace $M_j$, and let $(j_n)$ be any sequence taking values in $\{1,\dots,J\}$. Let $x_0 \in H$ be a vector, and let $(x_n)$ be the sequence defined by
\begin{equation*}
x_n = P_{j_n}x_{n-1}, \quad n\geq1.
\end{equation*}
Then $x_n$ converges weakly as $n\to \infty$.
\end{theorem}

In fact, Amemiya and Ando proved a slightly stronger result about contractions \cite{AmAn65}, of which Theorem \ref{amemiya} is a corollary. By proving Theorem \ref{amemiya} directly, we are able to simplify their proof. Additionally, by including details not originally present in \cite{AmAn65}, we aim to to make the proof easier to follow.

We present our proof through a series of four lemmas.
For simplicity, we write `neighbourhood' to mean a basic weakly open neighbourhood of $0$ in $H$. We also write $B_H$ for the closed unit ball in $H$.

\begin{lemma} \label{neighbourhood}
For any neighbourhood $U$, and any $j \in \{1,\dots,J\}$, there is an $\varepsilon = \varepsilon(j) > 0 $ such that for $x\in B_H$
\begin{equation*}
\|P_jx\| \geq 1-\varepsilon \implies (I-P_j)x \in U.
\end{equation*}
\end{lemma}
\begin{proof}
Let $U$ be a neighbourhood, and $x \in B_H$. Then there are $y_1,\dots ,y_r \in H$ and $\delta >0$, such that 
\begin{equation*}
U = \{x \in H: |\langle x,y_k\rangle| < \delta \textnormal{ for each }  1\leq k \leq r \}. 
\end{equation*}
Let $\varepsilon$ be small enough such that $0 < \eta \sqrt{2\varepsilon - \varepsilon^2}<\delta$, where $\eta  = \max \{\|y_k\| : 1\leq k \leq r \}$. For example, $\varepsilon = \min \{1, \frac{\delta^2}{2\eta ^2}\}$ works. Suppose that $\|P_jx\| \geq 1-\varepsilon$. Then by Lemma \ref{foundational}\ref{projection equality},
\begin{equation*}
\begin{aligned}
\|(I-P_j)x\|^2 = \|x-P_jx\|^2 = \|x\|^2 - \|P_jx\|^2 \leq 1 - (1-\varepsilon)^2 = 2\varepsilon - \varepsilon^2.
\end{aligned}
\end{equation*}
Hence by the Cauchy-Schwartz inequality, for any $k \in \{1,\dots ,r\}$,
\begin{equation*}
|\langle (I-P_j)x,y_k \rangle| \leq \|(I-P_j)x\| \|y_k\| \leq \eta \sqrt{2\varepsilon - \varepsilon^2} < \delta,
\end{equation*}
and thus $(I-P_j)x \in U$.
\end{proof}

\begin{lemma} \label{commutes}
Let $Q_j$ be the orthogonal projection onto $Ker(I-P_j\dots  P_1)$. Then for each $k \in \{1,\dots ,j\}$, $Q_j$ and $P_k$ commute.
\end{lemma}
\begin{proof}
By Lemma \ref{foundational 2}\ref{useful for amemiya}, for each $x\in H$, \[Q_j x \in \ker(I-P_j\dots P_1) = \bigcap_{k=1}^j\ker(I-P_k).\] Therefore for each $k \in \{1,\dots ,j\}$, we have $(I-P_k)Q_jx = 0$, and so $P_kQ_jx = Q_jx$. Hence, 
\begin{equation} \label{equality P Q}
P_kQ_j = Q_j.
\end{equation}
Since $P_k$ and $Q_j$ are self-adjoint, \[Q_j = (Q_j)^* = (P_kQ_j)^* = (Q_j)^*(P_k)^* = Q_jP_k,\] and so $P_kQ_j = Q_j = Q_jP_k$.
\end{proof}

\begin{lemma} \label{new neighbourhood}
Let $R_j = I - Q_j$. Then for any neighbourhood $U$, there is another neighbourhood $V$ such that for $x \in B_H$,
\begin{equation*}
(I-P_k)x \in V, \quad k \in \{1,\dots ,j\} \implies R_jx \in U.
\end{equation*}
\end{lemma}
\begin{proof}
Let $H^j$ be the Cartesian product $H\times \dots \times H$ ($j$ times), viewed as a Hilbert space with addition and scalar multiplication given by
\begin{equation*}
(u_1,\dots,u_j) + \lambda(v_1,\dots,v_j) = (u_1 + \lambda v_1, \dots, u_1 + \lambda v_j), \quad u_i,v_i \in H, \, \, \lambda \in \mathbb{F},
\end{equation*}
and inner product given by
\begin{equation*}
\big{\langle}(u_1,\dots,u_j),(v_1,\dots,v_j)\big{\rangle}_{H^j} = \langle u_1,v_1\rangle + \dots + \langle u_j,v_j\rangle, \quad u_i,v_i \in H.
\end{equation*}
Therefore the norm on $H^j$ is
\begin{equation*}
\|(u_1,\dots,u_j)\|_{H^j} = \sqrt{\|u_1\|^2 + \dots + \|u_j\|^2}, \quad u_1,\dots,u_j \in H.
\end{equation*}
We consider the map $g\colon R_j(H) \to H^j$ given by \[g(R_jx) = \big{(}(I-P_1)x,\dots ,(I-P_j)x\big{)}, \quad x \in H,\] where both spaces are endowed with their respective weak topologies. Lemma \ref{foundational 2}\ref{useful for amemiya} gives that for any $x\in H$,
\begin{equation*}
\begin{aligned}
(I-P_k)x = 0, \quad k \in \{1,\dots ,j\} &\iff x\in \bigcap_{k=1}^j \ker(I-P_k) \\
&\iff x\in \ker(I-P_j\dots  P_1) \\
&\iff Q_jx = x \\
&\iff R_jx=0.
\end{aligned}
\end{equation*}
The $\impliedby$ implication shows that $g$ is well defined, while the $\implies$ implication shows that $g$ is injective. We now note that (\ref{equality P Q}) gives
\begin{equation*}
(I-P_k)R_j = (I-P_k)(I-Q_j) = I-P_k - Q_j + P_kQ_j = I-P_k - Q_j + Q_j = I-P_k.
\end{equation*}
So given $x \in H$, and noting that $\|I-P_j\|\leq \|I\| + \|P_j\| \leq 2$, we have
\begin{equation*}
\begin{aligned}
\|g(R_jx)\|_{H^j} &= \| \big{(}(I-P_1)x,\dots ,(I-P_j)x\big{)} \|_{H^j} \\ 
&= \|\big{(}(I-P_1)R_jx,\dots ,(I-P_j)R_jx\big{)}\|_{H^j} \\ 
&= \sqrt{\|(I-P_1)R_jx\|^2 + \dots + \|(I-P_j)R_jx\|^2} \\
&\leq \|(I-P_1)R_jx\| + \dots + \|(I-P_j)R_jx\| \\
& \leq 2j\|R_jx\|.
\end{aligned}
\end{equation*}

Therefore $g$ is bounded. It is a simple check to see that $g$ is linear. Hence $g$ is continuous. Let $f$ be the restriction of $g$ to $R_j(B_H)$ (where $R_j(B_H)$ is endowed with the relative weak topology). Then $f$ must also be continuous and injective.

We know that a unit ball in a normed vector space $H$ is compact with respect to the weak topology if and only if $H$ is reflexive. In our case, $H$ is a Hilbert space, so is indeed reflexive. Therefore $B_H$ is weakly compact. Since $R_j$ is continuous, then $R_j(B_H)$ is also weakly compact. The weak topology on any vector space is Hausdorff, so in particular $H^j$ is Hausdorff with respect to the weak topology. Hence $f$ is an injective continuous map from a compact topological space into a Hausdorff space, so that \[\textnormal{ $R_j(B_H)$ and $f(R_j(B_H))$ are homeomorphic}.\] Therefore we can replace the codomain ($H^j$) of $f$ with the image of $f$ (endowed with the relative weak topology), so that $f$ becomes a bijection. In particular, $f^{-1}$ is then continuous at the origin, and hence the claim follows.
\end{proof}

For $j \in\{1,\dots ,J\}$, let $\mathcal{M}_j$ be the collection of maps which are in a free semigroup generated by $j$ of the projections $\{P_1,\dots,P_J\}$. We also set $\mathcal{M}_0 = \{I\}$.

\begin{lemma} \label{important neighbourhood lemma}
Let $U$ be a neighbourhood, and let $S \in \mathcal{M}_j$. There exists a positive number $\varepsilon = \varepsilon(U,j)$ depending only on $U$ and $j$ such that given $x \in B_H$, we have
\begin{equation*}
\|Sx\| \geq 1-\varepsilon \implies (I-S)x \in U.
\end{equation*}
\end{lemma}

\begin{proof}
We prove this by induction on $j$. The case $j=1$ follows immediately from Lemma \ref{neighbourhood} (since $S$ is just $P_i$ for some $i \in \{1,\dots ,J\}$). Suppose the assertion is true for $j-1$. Let $S \in \mathcal{M}_j$. If $S \in \mathcal{M}_{j-1}$, we would be done by the induction hypothesis. Therefore we may assume that \[S \in \mathcal{M}_j \setminus \mathcal{M}_{j-1}.\] Without loss of generality, we may also assume that $S$ is in the free semigroup generated by $P_1,P_2,\dots ,P_j$ (if not we simply relabel the projections). Then since $S \notin \mathcal{M}_{j-1}$, for any index $k \in \{1,\dots, j\}$, $S$ can be written in the form
\begin{equation*}
S=T_1P_kT_2 = T_3P_kT_4,
\end{equation*}
where $T_1,T_4 \in \mathcal{M}_{j-1}$, and $T_2,T_3 \in \mathcal{M}_{j}$. Let $U$ be a neighbourhood, and pick $V$ as in Lemma \ref{new neighbourhood}. Since $P_i$ is continuous, and so weakly continuous for each $i \in \{1,\dots ,j\}$, we may pick a neighbourhood $W$ such that
\begin{equation*}
4W + 4P_iW \subseteq V, \quad i \in \{1,\dots ,j\}.
\end{equation*}
Indeed, we have that $f_i = 4I + 4P_i$ is weakly continuous for each $i \in \{1,\dots,j\}$. Then since $V$ is weakly open, $f_i^{-1}(V)$ is too, and so we may find a neighbourhood $W_i$ contained in $f_i^{-1}(V)$. Hence,
\begin{equation*}
f_i(W_i) \subseteq f_i(f_i^{-1}(V))\subseteq V, \quad i \in \{1,\dots,j\}.
\end{equation*}
Now since $\bigcap_{i=1}^j W_i$ is a finite intersection of weakly open sets, it is itself weakly open. Therefore we many find a neighbourhood $W$ contained in $\bigcap_{i=1}^j W_i$. Then $W$ has the required property that for each $i \in \{1,\dots j\}$,
\begin{equation*}
4W + 4P_iW = f_i(W) \subseteq f_i(W_i) \subseteq V.
\end{equation*}
By the induction hypothesis, it is possible to find $\varepsilon_1$ such that for $x\in B_H$ and $T\in \mathcal{M}_{j-1}$, we have \[\|Tx\| \geq 1-\varepsilon_1 \implies (I-T)x \in W. \] By Lemma \ref{neighbourhood}, we can find $\varepsilon_2$ such that for $x\in B_H$ and $T=P_k$, \[\|Tx\| \geq 1-\varepsilon_2 \implies (I-T)x \in W.\] We set $\varepsilon = \min\{\varepsilon_1,\varepsilon_2\}$, and note it is independent of $S$ (since $\varepsilon_1$ and $\varepsilon_2$ are). Then given $x \in B_H$ and either $T\in \mathcal{M}_{j-1}$ or $T=P_k$,
\begin{equation} \label{neighbourhood implication}
\|Tx\| \geq 1-\varepsilon \implies (I-T)x \in W.
\end{equation}
We now fix $x \in B_H$, and we assume that $\|Sx\| \geq 1-\varepsilon$. If we can show that $(I-S)x \in U$, then our induction is complete. We note that
\begin{equation*}
1\geq \|T_4x\| \geq \|P_kT_4x\| \geq \|T_3P_kT_4x\| = \|Sx\| \geq 1-\varepsilon.
\end{equation*}
In particular, $\|T_4x\| \geq 1-\varepsilon$ and  $\|P_k(T_4x)\| \geq 1-\varepsilon$. Then by (\ref{neighbourhood implication}), we have $(I-T_4)x \in W$ and $(I - P_k)(T_4x) \in W$. Hence,
\begin{equation} \label{element of 1}
\begin{alignedat}{2}
(I-P_k)x &= (I-T_4)x + (I-P_k)(T_4x) - P_k(I-T_4)x \\
& \in W + W + P_kW \subseteq 2W + 2P_kW \subseteq \frac{1}{2} V.
\end{alignedat}
\end{equation}
We also note that $ \|T_1(P_kT_2x)\| = \|Sx\| \geq 1-\varepsilon$, so that by (\ref{neighbourhood implication}), we have $(I-T_1)(P_kT_2x) \in W$. Hence,
\begin{equation} \label{element of 2}
\begin{aligned}
(I-P_k)Sx &= (T_1 - I)P_kT_2x + P_k(I-T_1)P_kT_2x \\
&\in W + P_kW \subseteq \frac{1}{2} V. 
\end{aligned}
\end{equation}
Since (\ref{element of 1}) and (\ref{element of 2}) are valid for all $k \in \{1,\dots ,j\}$, Lemma \ref{new neighbourhood} guarantees that $R_jx \in \frac{1}{2} U$ and $R_j(Sx) \in \frac{1}{2} U$. Therefore, \[R_j(I-S)x \in U.\]
Recalling that we assumed $S \in \mathcal{M}_j \setminus \mathcal{M}_{j-1}$ is in the free semigroup generated by $P_1,P_2,\dots,P_j$, a similar argument to Lemma \ref{foundational 2}\ref{useful for amemiya}  gives
\begin{equation*}
\ker(I-P_j\dots P_1)=\bigcap_{k=1}^j \ker(I-P_k) = \ker(I-S). 
\end{equation*}
Since $Q_j$ is the projection onto  $\ker(I-P_j\dots  P_1) = \ker(I-S)$, we have that $(I-S)(I-R_j)x = (I-S)Q_jx  = 0$. Rearranging, we have \[(I-S)x = (I-S)R_jx.\] We note also that since $Q_j$ commutes with $P_k$ for each $k \in \{1,\dots ,j\}$ (by Lemma \ref{commutes}), then $R_j$ commutes with $P_k$ for each $k \in \{1,\dots ,j\}$. Therefore $R_j$ commutes with $S$, and so $R_j$ commutes with $I-S$. Hence
\begin{equation*}
(I-S)x = (I-S)R_jx = R_j(I-S)x \in U,
\end{equation*}
thus completing the induction. 
\end{proof}

We are finally able to prove that $(x_n)$ always converges with respect to the weak topology.

\begin{proof}[Proof of Theorem \ref{amemiya}]
We begin by noting that without loss of generality, we may assume that \[\|x_0\|=1.\] Indeed, if we prove Theorem \ref{amemiya} for this case, then we may extend it to any $x_0 \in H$ by noting that \[P_{j_n}\dots P_{j_1}x_0 \textnormal{ converges weakly } \iff P_{j_n}\dots P_{j_1}\frac{x_0}{\|x_0\|} \textnormal{ converges weakly}.\]

We also note that ($\|x_n\|$) is a monotonically decreasing sequence, bounded below by $0$, and so it converges to some non-negative limit. If $\|x_n\| \to 0$, then $x_n$ converges in norm, and hence converges weakly. We may therefore suppose that $\lim_{n\to \infty} \|x_n\| >0$, and so $\inf_{n\geq 0}\|x_n\|>0.$

For any neighbourhood $U$, let $\varepsilon = \varepsilon(U,J)$ be as in Lemma \ref{important neighbourhood lemma}. Then there exists an $N \in \mathbb{N}$ such that for $n \geq m \geq N$, we have \[\|x_n\| \geq (1-\varepsilon)\|x_m\|.\] 
Indeed, suppose for a contradiction there was no such $N$, so that for any $N_i\in \mathbb{N}$ there are $n_i \geq m_i \geq N_i$ such that $\|x_{n_i}\| < (1-\varepsilon)\|x_{m_i}\|$. We begin by picking $N_1 = 0$ and finding appropriate $n_1 \geq m_1 \geq N_1=0$. Then, letting $N_2 = n_1+1$, we pick $n_2 \geq m_2 \geq N_2$, and continue inductively in this way. We then have for $k \in \mathbb{N}$,
\begin{equation*}
\begin{aligned}
\|x_{n_k}\| &< (1-\varepsilon)\|x_{m_k}\| \leq (1-\varepsilon)\|x_{N_k}\| \leq (1-\varepsilon)\|x_{n_{k-1}}\| \\
& < (1-\varepsilon)^2\|x_{n_{k-2}}\| <\dots <(1-\varepsilon)^k\|x_{n_1}\| \\ 
&\leq (1-\varepsilon)^k\|x_0\| \to 0 \textnormal{ as } k\to \infty,
\end{aligned}
\end{equation*}
contradicting $\inf_{n\geq0}\|x_n\|>0$. 

For a given $U$, we find an $N$ as above, and let $n\geq m \geq N$. Let $x = \frac{x_m}{\|x_m\|}$, and note that there is some $S\in \mathcal{M}_J$ such that $x_n = S\circ x_m$. Then
\begin{equation*}
\|Sx\| = \frac{\|x_n\|}{\|x_m\|} \geq 1-\varepsilon,
\end{equation*}
so Lemma \ref{important neighbourhood lemma} guarantees that $(I-S)x \in U$. Hence
\begin{equation} \label{difference in neighbourhood}
\begin{aligned}
x_m - x_n &= (I-S)x_m = \|x_m\|(I-S) \frac {x_m} {\|x_m\|} \\
&= \|x_m\|(I-S)x \leq \|x_0\|(I-S)x \\
&= (I-S)x \in U.
\end{aligned}
\end{equation}
Let $\delta >0$ and $y \in H$. Since $U$ was an arbitrary neighbourhood, we can pick $U = \{x \in H : | \langle x,y \rangle | < \delta/2\}$.  Then (\ref{difference in neighbourhood}) gives that for $n \geq m \geq N$, we have $x_m - x_n \in U$, and therefore \[| \langle x_m - x_n, y \rangle | < \delta / 2.\] But we also note that $(x_n)$ is a bounded sequence ($\|x_n\| \leq \|x_0\|$ for each $n \in \mathbb{N}$), and so has a weakly convergent subsequence, say $x_{n_k}$, converging weakly to some limit $x_\infty \in H$. Then there exists some $K \geq N$ such that \[| \langle x_{n_K} - x_\infty, y \rangle | < \delta /2 .\] Therefore for $n\geq n_K$,
\begin{equation*}
|\langle x_n - x_\infty, y \rangle| \leq |\langle x_n - x_{n_{K}},y \rangle| + |\langle x_{n_{K}} - x_\infty,y \rangle| < \delta /2 + \delta /2 = \delta.
\end{equation*} 
Hence $x_n$ converges weakly to $x_\infty$, completing our proof.
\end{proof}

We have shown that $(x_n)$ always converges weakly, but we do not yet know to what limit. Rather than taking Amemiya and Ando's approach in finding the limit, we notice that a simpler argument due to Sakai (Lemma \ref{key result sakai 2}) \cite{Sak95} works here too. Combining this with Theorem \ref{amemiya}, we obtain the following.
\begin{corollary} \label{limit}
Suppose that $(j_n)$ takes every value in $\{1,\dots,J\}$ infinitely many times. Then $(x_n)$ converges weakly to the orthogonal projection of $x_0$ onto $M = \bigcap_{j=1}^J M_j$.
\end{corollary}

Since in a finite-dimensional Hilbert space,  convergence in norm is equivalent to weak convergence, the following corollary is immediate from Theorem \ref{amemiya}.

\begin{corollary}
Assume the same setting as in Theorem \ref{amemiya}, except that we now specify that $H$ is a finite-dimensional Hilbert space. Then $(x_n)$ converges in norm.
\end{corollary}

In fact, even more can be said. We end this section with a nice observation due to Sakai \cite{Sak95} (briefly mentioned in Section \ref{concluding remarks}). We consider the set $D=\{1\leq d \leq J : (j_n) \textnormal{ takes value $d$ infinitely many times}\}.$

\begin{lemma} \label{Sakai mistake}
Suppose there is some $i \in D$ such that $M_i$ is finite-dimensional. Then $(x_n)$ converges in norm.
\end{lemma}

Sakai gave a proof of this result in his paper \cite{Sak95}, but it appears to be incorrect. In particular, it seems as though Sakai assumed that if a sequence converges weakly to a limit, and has a subsequence which converges in norm, then the sequence converges in norm. However, this is not generally true, and we demonstrate this as follows.
 
It is known that there are sequences which converge weakly, but not in norm. By translating the sequence if needed, we may find a sequence $(a_n)$ converging weakly to $0$, but not in norm. But then the sequence $(b_n)$ given by
\begin{equation*}
(b_n) = (a_1,0,a_2,0,a_3,0 \dots)
\end{equation*}
converges weakly to $0$ and has a subsequence which converges in norm, but does not itself converge in norm. 

However, we find that Lemma \ref{Sakai mistake} still turns out to be true. We offer the following proof.

\begin{proof}[Proof of Lemma \ref{Sakai mistake}]
We know, due to Theorem \ref{amemiya}, that $(x_n)$ converges weakly to some limit $x_\infty$. Since $i \in D$, we may pass to a subsequence $(x_{n_k})_{k\geq 1}$ such that each $x_{n_k} \in M_i$, and note that it must also converge weakly to $x_\infty$ (as $k \to \infty$). However $M_i$ is finite-dimensional, and in a finite-dimensional space, weak convergence is equivalent to convergence in norm, so we have that $(x_{n_k})$ converges in norm to $x_\infty$. 

By definition of $D$, we may find some $t \in \mathbb{N}$ such that for $n\geq t$, we have $j_n \in D$. In particular, for $n\geq t$, we have $P_{j_n}x_\infty = x_\infty$. Since $(x_{n_k})$ converges in norm to $x_\infty$, for any $\varepsilon>0$, we may find some $K \geq t$ such that $\|x_{n_K} - x_\infty\| < \varepsilon$. So for $m\geq n_K$, we have
\begin{equation*}
\begin{aligned}
\|x_m - x_\infty\| &= \|P_{j_m}P_{j_{m-1}}\dots P_{j_{n_{K+1}}}x_{n_{K}} - x_\infty\| \\
& = \|P_{j_m}P_{j_{m-1}}\dots P_{j_{n_{K+1}}}(x_{n_{K}} - x_\infty)\| \\
& \leq \|x_{n_{K}} - x_\infty\| < \varepsilon.
\end{aligned}
\end{equation*}
Hence $(x_m)$ converges in norm to $x_\infty$, concluding our proof.
\end{proof}

\newpage

\section{Failure of strong convergence} \label{failure strong convergence}

As mentioned in the introduction, Amemiya and Ando's question as to whether there is a sequence of projections that does not converge in norm \cite{AmAn65} went unanswered for a long time. It was resolved only in 2012, when Paszkiewicz proved that for any infinite-dimensional Hilbert space, we may find five subspaces, a vector $x_0 \in H$, and a sequence $(j_n)$, so that $(x_n)$ does not converge in norm \cite{Pas12}. This construction was improved by Kopeck\'a and M\"uller from five subspaces to three \cite{KoMu14}, and then refined in 2017 by Kopeck\'a and Paszkiewicz \cite{KoPa17}.

In this section, we will closely follow Kopeck\'a and Paszkiewicz's construction, presenting a series of technical lemmas as stated in \cite{KoPa17}, leading to a proof of the following theorem.
\begin{theorem} \label{big result kopecka}
There exists a sequence $(j_n)$ with the following property. If $H$ is an infinite-dimensional Hilbert space, and $x_0 \in H$ is a non-zero vector, then there exists three closed subspaces $M_1,M_2,M_3 \subset H$ intersecting only at the origin such that the sequence $(x_n)$ does not converge in norm.
\end{theorem}
We aim to present the proofs of the lemmas in a more accessible way, adding additional details where they have been omitted. 
\subsection{Notation}

Given subsets $X,Y \subset H$, we will write $\bigvee X$ for the closed linear span of $X$, and $X \vee Y$ for the closed linear span of $X \cup Y$. We will also write $\bigvee_{i\in I} X_i$ for the closed linear span of $\bigcup_{i \in I} X_i$. Given $x,y \in H$, we will write $\vee x$ and $x\vee y$ for $\vee \{x\}$ and $\vee\{x,y\}$, respectively.

For $m \in \mathbb{N}$, we will write $\mathcal{S}_m$ to denote the free semigroup with generators $a_1, \dots, a_m$. If $A_1, \dots, A_m \in B(H)$, and $\varphi = a_{i_r}\dots a_{i_1} \in \mathcal{S}_m$ for some $r \in \mathbb{N}$ and $i_j \in \{1,\dots,m\}$, then we write
\begin{equation*}
\varphi(A_1,\dots,A_M) = A_{i_r}\dots A_{i_1} \in B(H).
\end{equation*} 
We refer to elements of a semigroup as `words', made up of the `letters' from the set $\{a_1,\dots,a_m\}$. We denote the length of the word $\varphi$ by $|\varphi| = r$, and the number of occurrences of the letter $a_i$ in the word $\varphi$ by $|\varphi _i|$, so that $\sum_{i=1}^m |\varphi_i| = |\varphi|=r$.

\subsection{Continuous dependence of words on letters}

We begin by proving that if we have control over the number of appearances of a contraction in a product, then replacing the contraction with one close in norm to the original does not change the product much.
 
 \begin{lemma} \label{contraction inequality}
 Let $\psi \in \mathcal{S}_n$ for some $n \in \mathbb{N}$. Assume that for each $i \in \{1,2,\dots ,n\}$, $A_i, B_i, E \in B(H)$ are contractions such that each $A_i$ commutes with $E$.  Then
 \begin{equation*}
 \|\psi(A_1,\dots ,A_n)E - \psi(B_1,\dots ,B_n)E\| \leq \sum_{1\leq i \leq n} |\psi_i| \cdot \|A_iE - B_iE\|. 
 \end{equation*}
 \end{lemma}
 
 \begin{proof}
 We prove this by induction on $|\psi|$. For $|\psi|=0$, we have that $\psi(A_1,\dots ,A_n) =  \psi(B_1,\dots ,B_n)$ is the identity on $H$, so the inequality holds (with both sides being $0$). For our induction hypothesis (IH), suppose the assertion is true for $|\psi| \leq r$. We now suppose $|\psi|= r+1$. Then we have $\psi = \varphi a_j$ for some $\varphi \in \mathcal{S}_n$ with $|\varphi| = r$ and $j\in \{1,2,\dots ,n\}$. Hence
 \begin{align*}
 &\|\psi(A_1,\dots ,A_n)E - \psi(B_1,\dots ,B_n)E\| && \\
 & =\|\varphi(A_1,\dots ,A_n)A_jE - \varphi(B_1,\dots ,B_n)B_jE\| \\
& \leq \|\varphi(A_1,\dots ,A_n)A_jE - \varphi(B_1,\dots ,B_n)A_jE\| \\
& \quad + \|\varphi(B_1,\dots ,B_n)A_jE - \varphi(B_1,\dots ,B_n)B_jE\| && \textnormal{\big{[}triangle inequality\big{]}} \\
& \leq \|\varphi(A_1,\dots ,A_n)E - \varphi(B_1,\dots ,B_n)E\| \\
& \quad + \|\varphi(B_1,\dots ,B_n)\| \cdot \|A_jE-B_jE\| && \textnormal{\big{[}$A_jE=EA_j$,\, $\|A_j\| \leq 1$\big{]}} \\
& \leq \|A_jE-B_jE\| + \sum_{1\leq i \leq n} |\varphi_i| \cdot \|A_iE - B_iE\| && \textnormal{\big{[}IH, $\|\varphi(B_1,\dots,  B_n)\| \leq 1$\big{]}} \\
& = \sum_{1\leq i \leq n} |\psi_i| \cdot \|A_iE - B_iE\|.
\end{align*}
This completes the induction.
 \end{proof}
 
We now state two simple corollaries of Lemma \ref{contraction inequality}, which will be useful later on.
 
 \begin{corollary} \label{contraction inequality corollary} Let $\psi \in \mathcal{S}_3$. Suppose $E,W,X,X',Y,Y',Z$ are subspaces of $H$, such that $W,X,Y \subset E$ and $X',Y' \perp$ E. Then
 \begin{equation*}
 \|\psi(P_W,P_X,P_Y)E - \psi(P_Z,P_{X\vee X'}, P_{Y\vee Y'})E\| \leq | \psi_1| \cdot \|P_ZP_E-P_W\|.
 \end{equation*}
 \end{corollary}
 
 \begin{proof}
We apply Lemma \ref{contraction inequality} for $n=3$, and note that $P_WP_E=P_W$, $P_X=P_XP_E=P_{X \vee X'}P_E$, and $P_Y=P_{Y}P_E=P_{Y \vee Y'}P_E$.
 \end{proof}
 
 \begin{corollary} \label{bound projection word}
 Let $n \in \mathbb{N}$, $\varphi \in S_n$ and $A_1,\dots,  A_n$, $B_1,\dots,  B_n$ be contractions. Then
 \begin{equation*}
\|\varphi(A_1,\dots ,A_n) - \varphi(B_1,\dots ,B_n)\| \leq |\varphi | \cdot \max_{1\leq i \leq n} \|A_i - B_i\|. 
\end{equation*}
 \end{corollary}
 
 \begin{proof}
 We apply Lemma \ref{contraction inequality}, with $E$ taken to be the identity map.
 \end{proof}
 
 \subsection{Constructing three subspaces and a finite product of projections} \label{triples}
 
 In this subsection, given orthonormal vectors $u$ and $v$ in $H$, we construct three subspaces $X$, $Y$, and $W=u \vee v$, and a finite product of projections $\psi(P_W,P_X,P_Y)$ onto $W$ such that $\psi(P_W,P_X,P_Y)u$ is close to $v$. This will be useful later when we `glue' together countably many copies of these triples of subspaces.

Let $\varepsilon >0$, and let $k=k(\varepsilon)$ be the smallest positive integer $k$ such that \[\Big{(}\cos \frac{\pi}{2k}\Big{)}^k > 1 - \varepsilon.\] We note that such a $k$ exists since $(\cos \frac{\pi}{2r})^r \to 1$ as $r \to \infty$. Let $u$ and $v$ be orthonormal vectors in $H$, and for $j \in \{0,\dots ,k\}$, let \[h_j=u\cos \frac{\pi j}{2k} + v\sin \frac{\pi j}{2k},\] so that $h_0 = u$ and $h_k=v$. 

Then by definition of $k$, if we project $u$ consecutively onto $\vee h_1,\dots ,\vee h_k$, then we arrive at $v$ with error less than $\varepsilon$, so that 
\begin{equation} \label{consecutive projections}
\|(P_{\vee h_k}\dots P_{\vee h_1})u - v\|<\varepsilon.
\end{equation}
This is illustrated in the following figure.
\vspace{4mm}
\begin{figure}[H]
\begin{center}
\includegraphics[width=81.82mm]{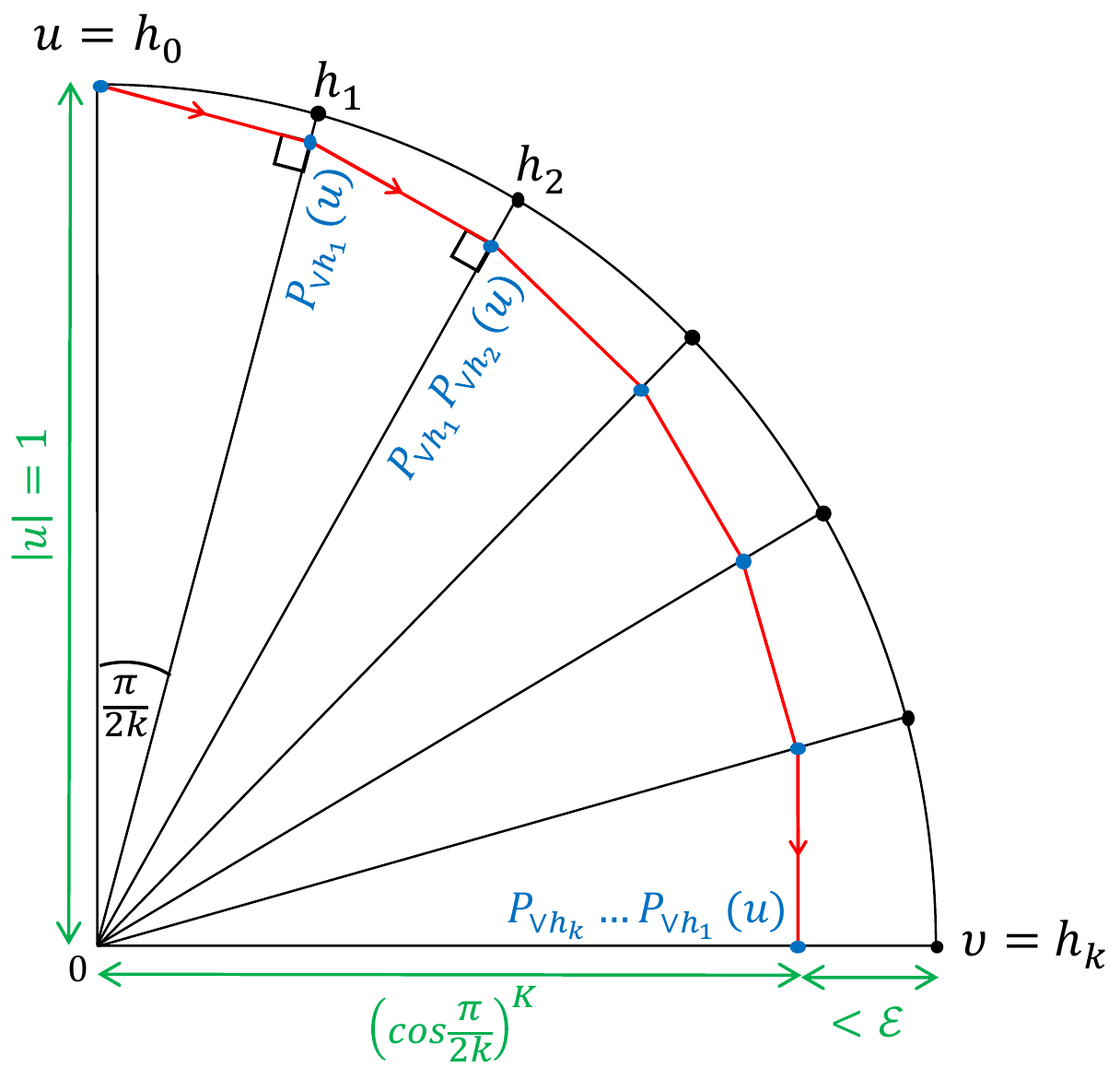}
\caption{Approximating $v$ by projections of $u$}
\label{segment projections}
\end{center}
\end{figure}
We have by Theorem \ref{von neumann} (von Neumann) \cite{von49} that a projection onto $\vee h_j$ can be arbitrarily well approximated by iterating projections between two subspaces with intersection $\vee h_j$. In the following lemma, we call these subspaces $W$ and $X_j'$, and show that we can approximate $X_j'$ with a subspace $X_j$ so that $X_1 \subset \dots \subset X_k$. We will then see in Lemma  \ref{replace projection} that we are able to replace the projections onto each $X_j$ by a product $(XYX)^{s(j)}$ where $X=X_k$, and $Y$ is such that $\|P_X - P_Y\|$ is small. 

Therefore, instead of projecting onto several subpaces to get from $u$ to $v$, we only need to project onto three of them, $W$, $X$, and $Y$, finitely many times. 
 
 \begin{lemma} \label{quarter circle}
 Let $\varepsilon >0$, and recall that $k=k(\varepsilon)$ is the smallest positive integer $k$ such that $(cos\frac{\pi}{2k})^k < 1 -\varepsilon$. There exists $\varphi \in S_{k+1}$ with the following property.
 
 Suppose $X$ is a subspace of $H$ with $\dim X = \infty$, and $u, v \in X$ are vectors such that $\|u\| = \|v\| = 1$ and $u \perp v$. Let $W=u \vee v$. Then there exists subspaces $X_1 \subset \dots  \subset X_{k(\varepsilon )} \subset X$ such that $\dim X_j = j+1$ for all $j \in \{1,\dots ,k \}$, and
 \begin{equation*}
 \| \varphi(P_W,P_{X_1},\dots ,P_{X_{k}})u - v\| < 2\varepsilon.
 \end{equation*}
 \end{lemma}
 
 \begin{proof}
We pick orthonormal vectors $z_0,z_1,\dots ,z_{k-1} \in W^{\perp} \cap X$. We then construct inductively a sequence $\alpha_1>\dots >\alpha_{k-1}>\alpha_k = 0$, and a sequence of subspaces $X_1 \subset \dots  \subset X_k \subset X$ in the following way. 

We pick $\alpha_0 \in (0,1)$ arbitrarily. Suppose that the sequence $\alpha_1>\dots > \alpha_{j-1}$, and the subspaces $X_1 \subset \dots  \subset X_{j-1}$ have already been constructed for some $j \in \{1,\dots, k-1\}$. We then set 
\begin{equation*}
X_j' = \bigvee \{h_0+\alpha_0z_0, h_1+\alpha_1z_1,\dots ,h_{j-1}+\alpha_{j-1}z_{j-1},h_j \}.
\end{equation*}
Since $z_i$ is orthogonal to $W$ for each $i\in \{1,\dots ,j-1\}$, we have \[W\cap X_j' = \vee h_j.\] Therefore by Theorem \ref{von neumann} (von Neumann), for each $x\in H$, we have
\begin{equation*}
(P_{X_j'}P_WP_{X_j'})^r x\to P_{\vee h_j}x \textnormal{ as } r\to\infty.
\end{equation*}
Since both $\vee h_j$ and $X_j'$ are finite-dimensional, both $P_{\vee h_j}$ and $P_{X_j'}P_WP_{X_j'}$ map into finite-dimensional subspaces of $H$. Hence there exists $r(j) \in \mathbb{N}$ such that
\begin{equation*}
\|(P_{X_j'}P_WP_{X_j'})^{r(j)} -  P_{\vee h_j}\| < \frac{\varepsilon}{k}. 
\end{equation*}
Let $X_j = \bigvee \{h_0+\alpha_0z_0, h_1+\alpha_1z_1,\dots ,h_{j-1}+\alpha_{j-1}z_{j-1},h_j + \alpha_jz_j\}$. As $\alpha_j \to 0$, $X_j$ is just a small perturbation of $X_j'$, so we can pick $\alpha_j>0$ small enough that
\begin{equation} \label{small perturbation}
\|(P_{X_j}P_WP_{X_j})^{r(j)} -  P_{\vee h_j}\| < \frac{\varepsilon}{k}.
\end{equation}
Suppose we have constructed $X_1 \subset \dots  \subset X_{k-1}$ and $\alpha_1>\dots >\alpha_{k-1}$ as above. We set $\alpha_k = 0$ and $X_k = X_k' =  \bigvee \{h_0+\alpha_0z_0, h_1+\alpha_1z_1,\dots ,h_{k-1}+\alpha_{k-1}z_{k-1},h_k\}$. We now find $r(k) \in \mathbb{N}$ such that (\ref{small perturbation}) holds also for $j=k$. Let $\varphi \in S_{k+1}$, and $\psi \in S_k$ be given by
\begin{align*}
\varphi (c,b_1,\dots ,b_k) &= (b_kcb_k)^{r(k)}\dots (b_1cb_1)^{r(1)}, \\
\psi(a_1,\dots ,a_k) &= a_1\dots  a_k.
\end{align*}
Then by Corollary \ref{bound projection word},
\begin{equation} \label{consecutive projections 2}
\begin{aligned}
&\| \varphi(P_W,P_{X_1},\dots ,P_{X_k}) - P_{\vee h_k}\dots  P_{\vee h_1} \| \\
&= \|\psi\big{(}(P_{X_k}P_WP_{X_k})^{r(k)},\dots ,(P_{X_1}P_WP_{X_1})^{r(1)}\big{)} - \psi(P_{\vee h_k}\dots  P_{\vee h_1})\| \\
& \leq |\psi|\cdot \max_{1\leq j \leq k} \|(P_{X_j}P_WP_{X_j})^{r(j)} - P_{\vee h_j} \| \\
&< k \cdot \frac{\varepsilon}{k} = \varepsilon.
\end{aligned}
\end{equation}
Hence, (\ref{consecutive projections}) and (\ref{consecutive projections 2}) give that
\begin{equation*}
\begin{aligned}
&\|\varphi(P_W,P_{X_1},\dots ,P_{X_k})u - v\| \\
&\leq \| \varphi(P_W,P_{X_1},\dots ,P_{X_k})u - (P_{\vee h_k}\dots P_{\vee h_1})u\| + \|(P_{\vee h_k}\dots P_{\vee h_1})u - v\| \\
&< \varepsilon + \varepsilon = 2\varepsilon,
\end{aligned}
\end{equation*}
completing the proof. We note that $\varphi$ does not depend on $X$, $u$, or $v$.
\end{proof}

We have constructed above a family of $k$ finite-dimensional subspaces. We see in the next lemma that these can be replaced by projections onto just two subspaces: the largest subspace in the family and a small variation of it. 

\begin{lemma} \label{replace projection}
Let $k \in \mathbb{N}$, $\varepsilon >0$, $\eta >0$, and $a>0$ be given. There exist natural numbers $a<s(k)<s(k-1)<\dots <s(1)$ with the following property. 

Suppose $X_1 \subset \dots  \subset X_k \subset X \subset E$ are closed subspaces of $H$, with $X$ separable and $\dim (X^\perp \cap E) = \infty$. Then there exists a closed subspace $Y \subset E$ such that $X \cap Y = \{0\}$, $\|P_X-P_Y\|<\eta$, and
\begin{equation*}
\|(P_XP_YP_X)^{s(j)} - P_{X_j}\| < \varepsilon, \quad j \in \{1,\dots ,k\}.
\end{equation*}
\end{lemma}

\begin{proof}
We may assume that $0<\eta<1$; if the statement holds in this case, then it clearly holds for any $\eta >0$. We begin by fixing $0<\beta_{k+1}<\frac{\eta}{2}$, and choosing $s(k) > a$ large enough that $\frac{1}{(1+\beta_{k+1}^2)^{s(k)}} < \varepsilon$. We then inductively choose numbers $\beta_k,s(k-1),\beta_{k-1},s(k-2),\dots ,s(1),\beta_1$ such that
\begin{equation} \label{inductive numbers}
\begin{gathered}
\beta_{k+1} > \beta_k >\dots >\beta_1 > 0, \\
a < s(k) < s(k-1) < \dots  < s(1), \\
\frac{1}{(1+\beta_{j+1}^2)^{s(j)}} < \varepsilon \textnormal{ \, and \, } \Big{|}\frac{1}{(1+\beta_{j}^2)^{s(j)}} - 1\Big{|}< \varepsilon, \quad j \in \{1,\dots ,k\}.
\end{gathered}
\end{equation}
We will show that these $s(j)$'s are as required. 
Since $X_1, \dots, X_k$ and $X$ are closed subspaces of a separable Hilbert space $H$, they are themselves separable Hilbert spaces under the same norm. A Hilbert space is separable if and only if it has an, at most, countable orthonormal basis. Hence we can find an, at most, countable orthonormal basis $\{e_i\}_{i\in I}$ in $X$ such that there are sets $\emptyset = I_0 \subset I_1 \subset \dots  \subset I_k \subset I_{k+1} = I$ with the property that $\{e_i\}_{i\in I_j}$ is an orthonormal basis in $X_j$ for $j \in \{1,\dots ,k\}$. 

For $e_i \in I_j \setminus I_{j-1}$ we define $\gamma_i = \beta_j$. Since $\dim (X^\perp \cap E) = \infty$, we can find a set of orthonormal vectors $\{w_i\}_{i\in I}$ in $X^\perp \cap E$ indexed by $I$. Let $Y = \bigvee \{e_i + \gamma_iw_i : i\in I\}$. Then for $i \in I$, it is a simple check that
\begin{equation*}
e_i = \underbrace{\frac{e_i + \gamma_iw_i}{{1+\gamma_i^2}}}_{\in Y} + \underbrace{e_i - \frac{e_i + \gamma_iw_i}{1+\gamma_i^2}}_{\in Y^\perp}. 
\end{equation*}
Hence $P_Ye_i = \frac{e_i + \gamma_iw_i}{1+\gamma_i^2}$. Therefore, since $e_i \in X$ and $w_i \in X^\perp$,
\begin{equation*}
(P_XP_YP_X)e_i = P_XP_Y(P_Xe_i) = P_X(P_Ye_i) = P_X(\frac{e_i + \gamma_iw_i}{1+\gamma_i^2}) = \frac{e_i}{1+\gamma_i^2}.
\end{equation*}
Hence, we have
\begin{equation*}
(P_XP_YP_X)^me_i = \frac{e_i}{(1+\gamma_i^2)^m}, \quad m\in\mathbb{N}.
\end{equation*}
Let $x \in X$. Writing it as $x = \sum_{i\in I} a_ie_i$, we see that Lemma \ref{foundational}\ref{pythagoras}, $\|e_i\| = 1$, (\ref{inductive numbers}), and $\gamma_i = \beta_j$, give
\begin{equation} \label{projections bound X}
\begin{aligned}
&\|(P_XP_YP_X)^{s(j)}x - P_{X_j}x\|^2 \\
&= \Big{\|} \sum_{i\in I} a_i \frac{e_i}{(1+\gamma_i^2)^{s(j)}} - \sum_{i\in I_j} a_ie_i \Big{\|}^2 \\
&= \Big{\|} \sum_{i\in I_j} a_ie_i\Big{(}-1+\frac{1}{(1+\gamma_i^2)^{s(j)}}\Big{)} +  \sum_{i\in I\setminus I_j} a_i\frac{e_i}{(1+\gamma_i^2)^{s(j)}} \Big{\|}^2 \\
&= \sum_{i\in I_j} \Big{\|}a_ie_i\Big{(}-1+\frac{1}{(1+\gamma_i^2)^{s(j)}}\Big{)}\Big{\|}^2 +  \sum_{i\in I\setminus I_j} \Big{\|}a_i\frac{e_i}{(1+\gamma_i^2)^{s(j)}} \Big{\|}^2 \\
&= \sum_{i\in I_j} |a_i|^2 \Big{(}1- \frac{1}{(1+\gamma_i^2)^{s(j)}}\Big{)}^2 + \sum_{i\in I\setminus I_j} |a_i|^2 \frac{1}{(1+\gamma_i^2)^{2s(j)}} \\
&\leq \sum_{i\in I_j} |a_i|^2  \varepsilon^2 + \sum_{i\in I\setminus I_j} |a_i|^2\varepsilon^2 \\
&= \varepsilon^2 \sum_{i\in I} |a_i|^2 = \varepsilon^2 \|x\|^2.
\end{aligned}
\end{equation}
We note that $P_{X_j}P_X = P_{X_j}$ (since $X_j \subset X$), and recall that projections are idempotent. Hence, by (\ref{projections bound X}), we have that for any $z \in H$ and $j \in \{1,\dots ,k\}$, 
\begin{align*}
\|(P_XP_YP_X)^{s(j)}z - P_{X_j}z\|^2 &= \|(P_XP_YP_X)^{s(j)}(P_Xz) - P_{X_j}(P_Xz)\|^2 \\
&\leq \varepsilon^2 \|P_Xz\|^2.
\end{align*}
Therefore,
\begin{equation*}
\|(P_XP_YP_X)^{s(j)} - P_{X_j}\| < \varepsilon, \quad j \in \{1,\dots ,k\}.
\end{equation*}
It remains to verify that $\|P_X-P_Y\| < \eta$, and that $X \cap Y = \{0\}$. For the latter, suppose that $z \in X \cap Y$. Then since $z \in X$, $z$ can be written as $\sum_{i\in I} b_ie_i$, and since $z \in Y$, $z$ can be written as $\sum_{i\in I} c_i(e_i+\gamma_iw_i)$ (where each $b_i,c_i \in \mathbb{F}$). Since each $w_i \in X^\perp$, and each $e_i \in X$, then for every $j \in I$,
\begin{equation*}
b_j = \langle z, e_j \rangle = \langle \sum_{i\in I} c_i(e_i+\gamma_iw_i), e_j \rangle = \langle \sum_{i\in I} c_ie_i,e_j \rangle = c_j.
\end{equation*}
Therefore, \[0 = z - z = \sum_{i\in I} b_i(e_i+\gamma_iw_i) - \sum_{i\in I} b_ie_i = \sum_{i\in I} b_i\gamma_iw_i.\] Hence $b_i$ = 0 for every $i \in I$, and so $z = \sum_{i\in I} b_ie_i = 0$. 

Finally, we want to show that $\|P_X-P_Y\| < \eta$. It is known that if $U = \bigvee \{f_i : i\in I\}$ for some orthonormal set $\{f_i\}_{i \in I}$, we have $P_Ux = \sum_{i\in I} \langle x,f_i \rangle f_i $ for all $x \in Z$. This, along with $0<\gamma_i<\beta_k+1<\frac{\eta}{2}<1$ and $|a-b|^2 \leq 2|a|^2 + 2|b|^2$, gives that for any $0\neq z \in H$,
\begin{equation*}
\begin{aligned}
&\|P_Xz - P_Yz\|^2 \\
&= \Big{\|} \sum_{i\in I} \langle e_i,z \rangle e_i - \sum_{i\in I} \frac{e_i + \gamma_iw_i }{1+\gamma_i^2} \langle e_i+\gamma_iw_i,z \rangle \Big{\|}^2 \\
&\leq \sum_{i\in I} \frac{1}{(1+\gamma_i^2)^2} \big{|}\gamma_i^2\langle e_i,z\rangle - \gamma_i \langle w_i,z \rangle \big{|}^2 + \frac{\eta^2}{4} \sum_{i\in I} \big{|}\frac{1}{1+\gamma_i^2} \langle e_i + \gamma_iw_i,z \rangle \big{|}^2 \\
&\leq 2(\eta/2)^4\|z\|^2 + 2(\eta /2)^2\|z\|^2 + (\eta^2/4)\|z\|^2 \\ 
&< \eta^2 \|z\|^2.
\end{aligned}
\end{equation*}
So indeed $\|P_X-P_Y\| < \eta$.
\end{proof}

As before, let $u$ and $v$ be orthonormal vectors, $W=u\vee v$, and $\varepsilon >0$. We proceed to make use of Lemmas \ref{quarter circle} and \ref{replace projection} to find a word $\psi$, and two (almost parallel) subspaces $X$ and $Y$, such that $\|\psi(P_W,P_X,P_Y)u - v\| < 3\varepsilon$.

\begin{lemma} \label{word and almost parallel subspaces}
For every $\varepsilon >0$, there exists $N=N(\varepsilon)$, such that for every $\eta>0$, there exists $\psi \in S_3$ with $|\psi_1| \leq N$ that has the following property.

Let $X\subset E$ be subspaces of $H$ such that $X$ is separable and $\dim (X^\perp \cap E) = \infty$. Let $u,v \in X$ be vectors such that $\|u\|=\|v\|=1$ and $u\perp v$. Let $W = u \vee v$. Then there exists a subspace $Y \subset E$ such that $X \cap Y = \{0\}$, $\|P_X-P_Y\|<\eta$, and
\begin{equation*}
\|\psi(P_W,P_X,P_Y)u - v\| < 3\varepsilon.
\end{equation*}
\end{lemma}

\begin{proof}
Let $\varepsilon >0$ and $\eta >0$ be given. Let $\varphi \in S_{k(\varepsilon)+1}$ be as in Lemma \ref{quarter circle}, and let $N = |\varphi_1|$.

Since $\|u\|=\|v\|=1$ and $u\perp v$, we can apply Lemma \ref{quarter circle} to see that there exist subspaces $X_1 \subset \dots  \subset X_{k(\varepsilon)} \subset X$ such that
\begin{equation*}
\| \varphi(P_W,P_{X_1},\dots ,P_{X_{k(\varepsilon)}})u - v\| < 2\varepsilon.
\end{equation*}
For $k=k(\varepsilon)$, the given $\eta$, $a=1$, and $\varepsilon$ replaced by $\frac{\varepsilon}{|\varphi|}$, we choose natural numbers $s(k) < s(k-1) < \dots  < s(1)$ as in Lemma \ref{replace projection}. Since $X$ is separable and $\dim X^\perp \cap E = \infty$, Lemma \ref{replace projection} gives that there exists a subspace $Y$ of $E$ such that $X \cap Y = \{0\}$, $\|P_X-P_Y\|<\eta$ and for each $j\in \{1,\dots ,k\}$,
\begin{equation*}
\|(P_XP_YP_X)^{s(j)} - P_{X_j}\| < \frac{\varepsilon}{|\varphi|}.
\end{equation*}
We then define $\psi$ to be $\varphi$, but with $a_i$ replaced by $(a_2a_3a_2)^{s(i-1)}$ for each $i\in \{2,\dots ,k+1\}$, so that
\begin{equation*}
\psi(P_W,P_X,P_Y) = \varphi(P_W,(P_XP_YP_X)^{s(1)},\dots ,(P_XP_YP_X)^{s(k)}).
\end{equation*}
It is simple to see that $|\psi_1| = |\varphi_1| = N$. Finally, by Corollary \ref{bound projection word},
\begin{align*}
& \|\psi(P_W,P_X,P_Y)u - v\| \\
&= \|\varphi(P_W,(P_XP_YP_X)^{s(1)},\dots ,(P_XP_YP_X)^{s(k)})u - v\| \\
& \leq \| \varphi(P_W,(P_XP_YP_X)^{s(1)},\dots ,(P_XP_YP_X)^{s(k)})u - \varphi(P_W,P_{X_1},\dots ,P_{X_k})u \| \\
& \quad + \|\varphi(P_W,P_{X_1},\dots ,P_{X_k})u - v\| \\
& \leq |\varphi| \cdot \frac{\varepsilon}{|\varphi|} + 2\varepsilon = 3\varepsilon.
\end{align*}
This concludes the proof.
\end{proof}

We may now make use of Corollary \ref{contraction inequality corollary} to show that we in fact have some freedom in our choice of  $W$, $X$, and $Y$ above.

\begin{lemma} \label{freedom choice spaces}
For every $\varepsilon>0$, there exists $\delta = \delta(\varepsilon)$ such that for every $\eta >0$, there exists $\psi \in S_3$ with the following property.

Let $X\subset E$ be subspaces of $H$ such that X is separable and $\dim X=\dim X^\perp \cap E = \infty$. Let $u,v \in X$ be vectors such that $\|u\|=\|v\|=1$ and $u\perp v$. Let $W=u\vee v$. Then there exists a subspace $Y\subset E$ such that $X \cap Y = \{0\}$ and $\|P_X - P_Y\|<\eta$ with the following property. If $X',Y',Z$ are subspaces such that $X',Y' \subset E$ and $\|P_W-P_ZP_E\|<\delta$, then
\begin{equation*}
\|\psi(P_Z,P_{X\vee X'},P_{Y\vee Y'})u - v\| < 4\varepsilon. 
\end{equation*}
\end{lemma}

\begin{proof}
Given $\varepsilon >0$, we pick $N \in \mathbb{N}$ as in Lemma \ref{word and almost parallel subspaces}, and let $\delta = \frac{\varepsilon}{N}$. For these $\varepsilon$ and $N$, and a given $\eta > 0$, we choose $\psi$ according to  Lemma \ref{word and almost parallel subspaces}. For a given subspace $X$, we also choose $Y$ according to this lemma. Let $X',Y',Z$ be as above. Then applying both Corollary \ref{contraction inequality corollary} and Lemma \ref{word and almost parallel subspaces}, we have
\begin{align*}
& \|\psi(P_Z,P_{X\vee X'},P_{Y\vee Y'})u - v\| \\
& \leq \|\psi(P_Z,P_{X\vee X'},P_{Y\vee Y'})u - \psi(P_W,P_X,P_Y)u\| + \|\psi(P_W,P_X,P_Y)u - v\| \\
&\leq |\psi_1| \cdot \|P_W-P_ZP_E\| + 3\varepsilon \\
&\leq N\delta + 3\varepsilon = 4\varepsilon. \qedhere
\end{align*}
\end{proof}

\subsection{`Gluing' together the triples}

The last step in proving Theorem \ref{big result kopecka} uses Lemma \ref{freedom choice spaces} to show that given an orthonormal set $\{e_i\}_{i=1}^\infty$ with an infinite-dimensional orthogonal complement, we can construct three closed subspaces $X,Y,Z$ of $H$ and words $\Psi^{(i)}$ such that $\Psi^{(i)}(P_Z,P_X,P_Y)e_i$ is close to $e_{i+1}$ for every $i \in \mathbb{N}$. Kopeck\'a and Paszkiewicz refer to this as `gluing' together countably many of the triples $W$, $X$, and $Y$ considered in Section \ref{triples} \cite{KoPa17}.

\begin{lemma} \label{almost there eva}
For any $\varepsilon_i >0$ where $i\in\mathbb{N}$, there exists $\Psi^{(i)} \in \mathcal{S}_3$ with the following property.

Suppose $\{e_i\}_{i=1}^{\infty}$ is an orthonormal set in $H$ with an infinite-dimensional orthogonal complement. Then there are three closed subspaces $X,Y,Z$ of $H$ such that
\begin{equation} \label{Psi projection inequality}
\|\Psi^{(i)}(P_Z,P_X,P_Y)e_i - e_{i+1}\| < 4\varepsilon_i, \quad i\in \mathbb{N}.
\end{equation}
\end{lemma}

\begin{proof}
For each $\varepsilon_i >0$ ($i \in \mathbb{N}$), we define $\delta_i = \delta(\varepsilon_i)$ as in Lemma \ref{freedom choice spaces}. We set $\delta_0 = 1$ and $\eta_i = \min\{\delta_{i-1},\delta_{i+1}\}$, and choose $\psi^{(i)} \in S_3$ as in Lemma \ref{freedom choice spaces}. We define $\Psi^{(i)}$ as follows,
\begin{equation*}
\Psi^{(i)}(P_Z,P_X,P_Y) =
\begin{cases}
\psi^{(i)}(P_Z,P_X,P_Y) & \text{if $i$ is even},\\
\psi^{(i)}(P_Y,P_X,P_Z) & \text{if $i$ is odd}.
\end{cases}
\end{equation*}
We begin by finding, for each $i \in \mathbb{N}$, closed infinite-dimensional subspaces $E_i$ of $H$ such that
\begin{equation} \label{finding ei}
\begin{gathered}
e_i,e_{i+1} \in E_i, \\
P_{\vee e_{i+1}} = P_{E_i}P_{E_{i+1}} = P_{E_{i+1}}P_{E_i}, \\
E_i \perp E_j \textnormal{ if } |i-j| \geq 2.
\end{gathered}
\end{equation}
We first note that since $\{e_i\}_{i=1}^\infty$ has an infinite-dimensional orthogonal complement, we may find an orthonormal set $\{f_i\}_{i=1}^\infty$, such that $e_i \perp f_i$ for every $i,j \in \mathbb{N}$. We then consider the infinite-dimensional spaces \[F_k = \bigvee \big{\{}f_i : i=(p_k)^r \textnormal{ for some } r\in\mathbb{N} \big{\}} ,\] where $p_k$ is the $k\textsuperscript{th}$ prime number. We set \[E_{i} = \langle e_{i} \rangle \oplus \langle e_{i+1} \rangle \oplus F_{i},\] and note these $E_i$ do indeed satisfy (\ref{finding ei}).

For each $i \in \mathbb{N}$, we find a closed subspace $X_i \subset E_i$, such that $e_i,e_{i+1} \in X_i$, and $\dim X_i = \dim(X_i^\perp \cap E_i) = \infty$. We then have $X_n = \langle e_{n} \rangle \oplus \langle e_{n+1} \rangle \oplus \widetilde{F_{n}}$ for some closed infinite-dimensional subspace $\widetilde{F_{n}} \subset F_n$, and
\begin{equation*} \label{subspace conditions}
\begin{gathered}
e_i,e_{i+1} \in X_i, \\
P_{\vee e_{i+1}} = P_{X_i}P_{X_{i+1}} = P_{X_{i+1}}P_{X_i}, \\
X_i \perp X_j \textnormal{ if } |i-j| \geq 2.
\end{gathered}
\end{equation*}
By Lemma \ref{freedom choice spaces}, there exist closed subspaces $Y_i \subset E_i$ such that $\|P_{X_i} - P_{Y_i}\| < \eta_i$, and
\begin{equation} \label{above lemma property}
\|\psi^{(i)}(P_{Z_i},P_{X_i\vee X'},P_{Y_i\vee Y'})e_i - e_{i+1}\| < 4\varepsilon,
\end{equation}
whenever $W_i = e_i\vee e_{i+1}$, and $X',Y',Z_i$ are subspaces such that $X',Y' \subset E_i^\perp$ and $\|P_{W_i} - P_{Z_i}P_{E_i}\|<\delta$. 

We now set $Y_0 = \vee e_1$ and
\begin{equation*}
X = \bigvee_{i \in \mathbb{N}} X_i, \quad Y= \bigvee_{i\in\mathbb{N}_{\geq0}} Y_{2i}, \quad Z = \bigvee_{i\in\mathbb{N}_{\geq0}} Y_{2i+1}.
\end{equation*}
Then setting \[X_i' = \widetilde{F_{i-1}} \vee \widetilde{F_{i+1}} \vee \bigvee_{\substack{j\in\mathbb{N} \\ j\notin \{i-1,i+1\}}}X_j,\]we have $X_i' \perp E_i$ and $X = X_i \vee X_i'$. We proceed to show (\ref{Psi projection inequality}) by considering the cases where $i$ is even and odd separately. 

Suppose first that $i$ is even. Then as above, for each $i \in \mathbb{N}$, we can find a subspace $Y_i'$ of $H$ such that $Y_i' \perp E_i$ and $Y=Y_i \vee Y_i'$. 

We note that $P_ZP_{E_i} = P_{Y_{i-1} \vee Y_{i+1}}P_{E_i}$, $P_{W_i} = P_{X_{i-1} \vee X_{i+1}}P_{E_i}$, $X_{i-1} \perp X_{i+1}$, and $Y_{i-1}\perp Y_{i+1}$. By Lemma \ref{foundational}\ref{sum projections closed}, for orthogonal closed subspaces $U,V$ of $H$, we have that $U+V = U\vee V$.  Applying this, along with Lemma \ref{foundational}\ref{adding projections}, we have
\begin{equation*}
\begin{aligned}
\|P_{W_i} - P_{Z}P_{E_i}\| &= \|P_{X_{i-1} \vee X_{i+1}}P_{E_i} - P_{Y_{i-1} \vee Y_{i+1}}P_{E_i}\| \\
&= \|P_{X_{i-1} + X_{i+1}}P_{E_i} - P_{Y_{i-1} + Y_{i+1}}P_{E_i}\| \\
&= \|(P_{X_{i-1}} + P_{X_{i+1}})P_{E_i} - (P_{Y_{i-1}} + P_{Y_{i+1}})P_{E_i} \| \\
&= \|(P_{X_{i-1}} - P_{Y_{i-1}})P_{E_i} + (P_{X_{i+1}} - P_{Y_{i+1}})P_{E_i} \| \\
&\leq \|P_{X_{i-1}} - P_{Y_{i-1}}\| + \|P_{X_{i+1}} - P_{Y_{i+1}} \| < \eta_{i-1} + \eta_{i+1} \\
&= \min \{\delta_{i-2}, \delta_i \} + \min \{\delta_{i}, \delta_{i+2} \} \leq \delta_i.
\end{aligned}
\end{equation*}
Hence, by (\ref{above lemma property}),
\begin{equation*}
\|\Psi^{(i)}(P_{Z},P_{X},P_{Y})e_i - e_{i+1}\| = \|\psi^{(i)}(P_{Z},P_{X},P_{Y})e_i - e_{i+1}\| < 4\varepsilon_i.
\end{equation*}
If $i$ is odd, then we can find a subspace $Y_i'$ of $H$ such that $Y_i' \perp E_i$ and $Z=Y_i \vee Y_i'$. As above, we show that $\|P_{W_i} - P_{Y}P_{E_i}\| < \delta_i$, and so by (\ref{above lemma property}),
\begin{align*}
& \|\Psi^{(i)}(P_{Z},P_{X},P_{Y})e_i - e_{i+1}\| = \|\psi^{(i)}(P_{Y},P_{X},P_{Z})e_i - e_{i+1}\| < 4\varepsilon_i. \qedhere
\end{align*}
\end{proof}
We are finally able to prove Theorem \ref{big result kopecka}, that a sequence of alternating projections may diverge.

\begin{proof}[Proof of Theorem \ref{big result kopecka}]
For $\varepsilon_i = 9^{-i}$ ($i \in \mathbb{N}$), we pick $\Psi^{(i)}$ as in Lemma \ref{almost there eva}. Let $e_1 = \frac{x_0}{\|x_0\|}$. Since $H$ is infinite-dimensional, we can find an orthonormal set $\{e_i\}_{i=1}^\infty$ with an infinite-dimensional orthogonal complement. We choose closed subspaces $X, Y, Z$ as in Lemma \ref{almost there eva}, renaming them $M_1,M_2$ and $M_3$ respectively. Let $A_k = \Psi^{(k)}(P_{M_1},P_{M_2},P_{M_3})$. We then have, for all $k \in \mathbb{N}$,
\begin{equation} \label{KoPa last}
\begin{aligned}
&\|A_kA_{k-1}\dots A_1e_1 - e_{k+1}\| \\
& \leq \|A_kA_{k-1}\dots A_2(A_1e_1 - e_2)\| + \|A_kA_{k-1}\dots A_2 e_2- e_{k+1}\| \\
& < 4\varepsilon_1 + \|A_kA_{k-1}\dots A_2 e_2- e_{k+1}\| \\
& \leq 4\varepsilon_1 + \|A_kA_{k-1}\dots A_3(A_2e_2 - e_3)\| + \|A_kA_{k-1}\dots A_3e_3- e_{k+1}\| \\
& < 4\varepsilon_1 + 4\varepsilon_2 + \|A_kA_{k-1}\dots A_3e_3- e_{k+1}\| \\
&\vdotswithin{=} \\
&< 4\varepsilon_1 + 4\varepsilon_2 + \dots + 4\varepsilon_{k-1} + \|A_ke_k - e_{k+1}\| \\
&< 4(9^{-1} + \dots + 9^{-k}) < 4 \sum_{k=1}^{\infty}\frac{1}{9^j} = \frac{1}{2}.
\end{aligned}
\end{equation}
By construction, each $A_k$ is some product of orthogonal projections onto $M_1$, $M_2$ or $M_3$. Let $n_k$ be the total number of projections in the product $A_kA_{k-1}\dots A_1$. We define the sequence $(j_n)$ by letting $j_n$ take value $i$ whenever the  $n\textsuperscript{th}$ projection in $A_kA_{k-1}\dots A_1$ is onto $M_i$. We define the sequence $(x_n)$ as in the statement of the theorem, so that $x_{n_k} = A_k \dots A_1 x_0$.

We will now show that the subsequence  $(x_{n_k})_{k\geq1}$ does not converge in norm. This then implies that $(x_n)$ does not converge in norm either, and we are done. 

We define the sequence \[(y_k) = \Big{(}\frac{x_{n_k}}{\|x_0\|}\Big{)} = (A_k\dots A_1e_1).\] By (\ref{KoPa last}), for each $k \in \mathbb{N}$, we have \[ \|y_k-e_{k+1}\| < \frac{1}{2} .\] Applying the reverse triangle inequality and the Cauchy-Schwarz inequality (and noting that $\|e_{k+1}\|$=1), we have
\begin{equation*}
\begin{aligned}
|\langle y_k,e_{k+1} \rangle | &\geq |\langle e_{k+1},e_{k+1} \rangle| - |\langle y_k - e_{k+1},e_{k+1} \rangle| \\
&= 1 - |\langle y_k - e_{k+1},e_{k+1} \rangle| \\
&\geq 1 - \|y_k - e_{k+1}\| \\
&\geq 1- \frac{1}{2} = \frac{1}{2}.
\end{aligned}
\end{equation*}
Now suppose for a contradiction that $y_k$ converges in norm to some limit $y$. Then there is some $m \in \mathbb{N}$ such that for $k\geq m$, \[ \|y_k-y\| < \frac{1}{4}. \] Again, applying the reverse triangle inequality and Cauchy-Schwartz inequality, we have that for $k \geq m$,
\begin{equation*}
\begin{aligned}
|\langle y,e_{k+1} \rangle | &\geq |\langle y_k, e_{k+1} \rangle | - | \langle y - y_k,e_{k+1} \rangle| \\
& \geq |\langle y_k, e_{k+1} \rangle | - \|y-y_k\| \\
& > \frac{1}{2} - \frac{1}{4} = \frac{1}{4}.
\end{aligned}
\end{equation*}
We therefore have, by Bessel's inequality, that
\begin{equation*}
\|y\|^2 \geq \sum_{k=1}^\infty |\langle y,e_k \rangle|^2 \leq \sum_{k=m}^\infty |\langle y,e_{k+1} \rangle|^2 = \sum_{k=m}^\infty \frac{1}{16} = \infty,
\end{equation*}
a contradiction. Hence $(y_k)$ does not converge in norm. Therefore $(x_{n_k})$ does not converge in norm, and so neither does $(x_n)$. This completes the proof.
\end{proof}
 
\subsection{An extension}

In fact, Kopeck\'a and Paszkiewicz went on to prove that there exist three closed subspaces in $H$ such that for any non-zero vector $x_0 \in H$, there is some sequence of projections $(j_n)$ for which $(x_n)$ does not converge in norm  \cite{KoPa17}. In particular, here we begin by choosing three subspaces, and then given a non-zero vector $x_0 \in H$, we find an appropriate sequence $(j_n)$. This is in contrast with Theorem \ref{big result kopecka}, where we first find a sequence $(j_n)$, and then given a non-zero vector $x_0 \in H$, we find appropriate subspaces.

The main idea of the proof is showing the following. Suppose we have three closed subspaces $X_1,X_2,X_3 \subset H$, a non-zero vector $x_0 \in H$, and a sequence of projections $(j_n)$ such that $(x_n)$ does not converge in norm (which we know is possible by Theorem \ref{big result kopecka}). Then we may find a closed infinite-dimensional subspace $L$ of $H$ such that for every non-zero $y_0 \in L$, there exists a sequence $(k_n)$ taking values in $\{1,2,3\}$ such that the sequence given by
\begin{equation*}
y_n = P_{Y_{k_n}}y_{n-1}, \quad n\geq 1,
\end{equation*}
does not converge in norm, where $Y_i = X_i \cap L$ for $i \in \{1,2,3\}$.

The proof is technical, non-constructive, and not directly relevant to our focus, so we omit it.

We end the section with a brief remark. As mentioned in Section \ref{concluding remarks}, Sakai's paper \cite{Sak95} ends by posing the following open question. For an arbitrary sequence $s = (j_n)$, does (\ref{key inequality}) always hold with $A=J-1$? That is, does
\begin{equation*}
\|x_n - x_m\|^2 \leq A \sum_{k=m}^{n-1} \|x_{k+1} - x_k \|^2
\end{equation*}
hold with $A=J-1$, and any $n \geq m \geq 1$?

In light of Theorem \ref{big result kopecka},  this is easily resolved. We find a sequence $s$ for which (\ref{key inequality}) does not hold for any constant $A$.

\begin{corollary} \label{Sakai open question}
There exists a sequence s, and subspaces $M_1,M_2,M_3$ in $H$ such that there is no constant $A$ for which
\begin{equation*}
\|x_n - x_m\|^2 \leq A \sum_{k=m}^{n-1} \|x_{k+1} - x_k \|^2
\end{equation*}
holds for any $n>m\geq1$.
\end{corollary}

\begin{proof}
Let $s=(j_n)$ be the sequence in Theorem \ref{big result kopecka}. We pick a vector $x_0 \in H$, and choose $M_1$, $M_2$, and $M_3$ according to this theorem. Suppose for a contradiction there is such a constant $A$. Then by Lemma \ref{key result sakai}, we have that $x_n$ converges in norm, contradicting Theorem \ref{big result kopecka}.
\end{proof}
\newpage

\section{Concluding remarks}

There is a lot of interesting mathematics related to the method of alternating projections that we could not fit into this dissertation. Two main areas we have not covered are what happens when we have closed convex subsets instead of closed subspaces, and the rate of convergence in the method of alternating projections.

\subsection{Closed convex subsets}
There are many extensions of the method of alternating projections. These include considering closed convex subsets rather than closed subspaces, contractions rather than projections, or generalising results to Banach spaces with certain properties (for example, considering a uniformly convex Banach space instead of a Hilbert space; see \cite{BaLy10, BaSe17, BrRe77}). Here, we give a brief summary of known results when we have closed convex subsets.

We begin by remarking that we can indeed define a projection $P$ onto a closed convex subset $C$ of $H$. By the Hilbert projection theorem, for any $x \in H$, there exists a unique $y \in C$ minimising $\|x-y\|$ over $C$. We define the projection $P_C$ of $x$ onto $C$ by $P_C(x) = y$.

In 1965, Bregman proved that any sequence of periodic projections converges weakly to an element in the intersection of the closed convex subsets (assuming the intersection is non-empty) \cite{Bre65}. We note that the intersection of a finite number of closed convex subsets is also closed and convex. However, as opposed to the case of closed subspaces, the point we converge to need not be the projection onto the intersection of the closed convex subsets. We offer an example to illustrate this.

\begin{figure}[H]
\begin{center}
\includegraphics[width=100mm]{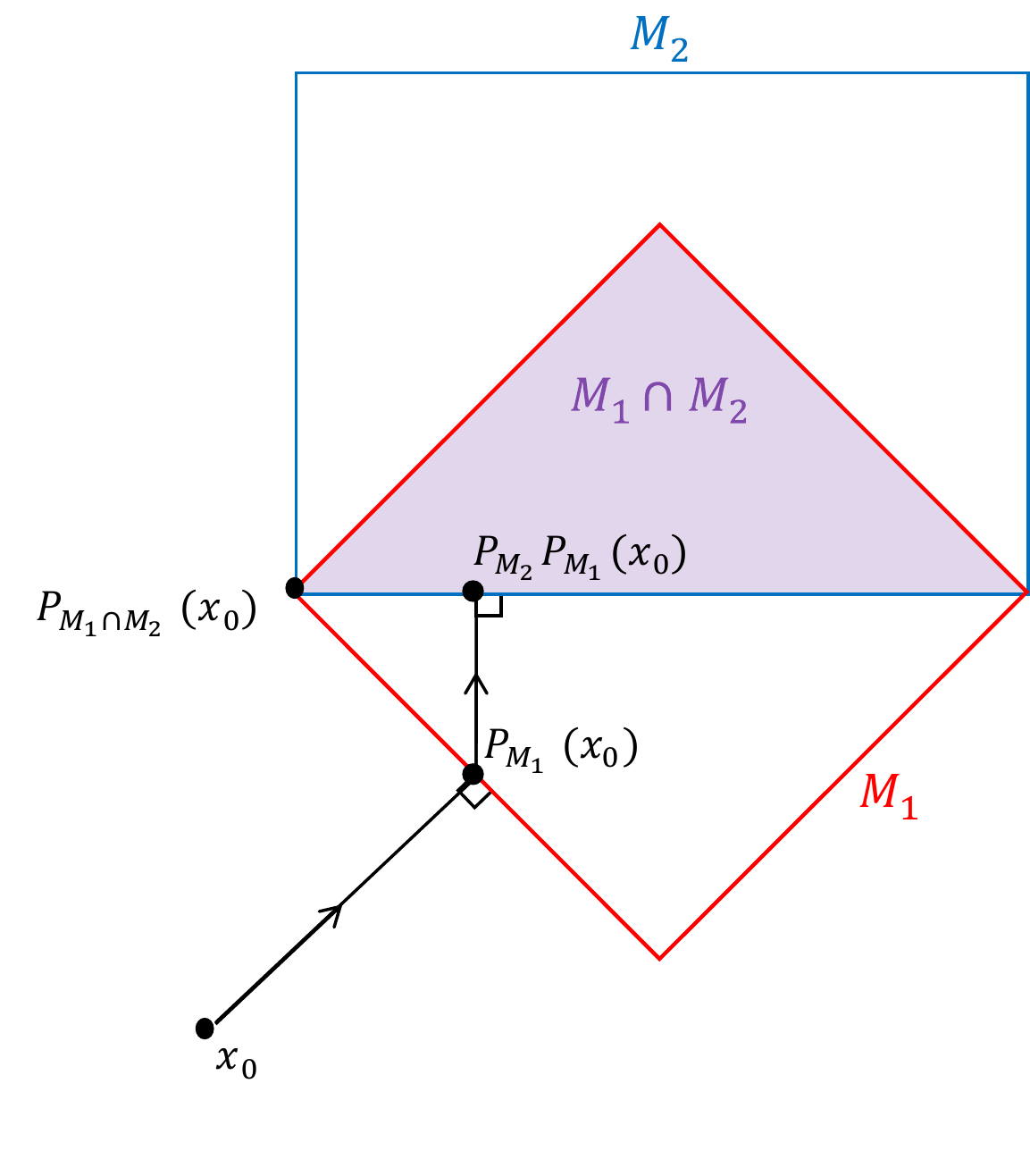}
\caption{An example of two closed convex subsets where $(x_n)$ does not converge to the projection of $x_0$ onto the intersection of the two subsets}
\label{intersection projection}
\end{center}
\end{figure}

We note that Bregman's result implies that we have convergence in norm for periodic projections when $H$ is finite-dimensional, since in this case, convergence in norm and weak convergence are equivalent. In fact, an identical argument to our proof of Lemma \ref{Sakai mistake} shows that it is enough for only one of the closed subsets to be contained in a finite-dimensional space.

As in the case for closed subspaces, it is natural to ask if we always have convergence in norm. In 2004, Hundal constructed an example of two closed convex subsets $C_1$ and $C_2$, intersecting only at the origin, such that $(P_{C_2}P_{C_1})^n$ converges weakly to $0$ (by Bregman's result), but does not converge in norm \cite{Hun04}. In fact, this proof was an important input towards Paszkiewicz's construction of five subspaces, a non-zero vector $x_0 \in H$, and a sequence $(j_n)$ such that $(x_n)$ does not converge in norm \cite{KoPa17, Pas12}.

\subsection{Rates of convergence}
Let $H=\mathbb{R}^2$, and let $\theta \in (0,\pi/2)$ be fixed. We consider the two closed subspaces 
\begin{align*}
M_1 &= \big{\{}(x,y) \in \mathbb{R}^2 : x=y\big{\}}, \\
M_2 &= \big{\{}(t\cos \theta,t \sin \theta) \in \mathbb{R}^2 : t \in \mathbb{R}\big{\}}.
\end{align*}
Our example in the introduction is the case $\theta = \pi/4$. Looking at Figure \ref{alternating projections}, it is not surprising that if we increase the angle $\theta$, we converge faster to $M = M_1\cap M_2 = \{0\}$. 
What may be more surprising is that we may extend this idea to define the notion of an angle between subspaces. 

The Friedrichs angle between two closed subspaces $M_1$ and $M_2$ of $H$ is defined to be the angle in $[0,\frac{\pi}{2}]$ whose cosine is given by
\begin{equation*}
c(M_1,M_2) = \sup \{ |\langle x,y \rangle | : x \in M_1 \cap M^\perp \cap B_H, y\in M_2 \cap M^\perp \cap B_H \}.
\end{equation*}
It is known that for all $n\geq1$, we have
\begin{equation*}
\|(P_{M_2}P_{M_1})^n - P_M\| = c(M_1,M_2)^{2n-1}.
\end{equation*}
The upper bound was proved by Aronszajn \cite{Aro50}, and equality by Kayalar and Weinert \cite{KaWe88}. Hence, letting $T=P_{M_2}P_{M_1}$, we have that $T^n$ converges uniformly (in operator norm) to $P_M$ if and only if $c(M_1,M_2)<1$ (i.e. the Friedrichs angle between $M_1$ and $M_2$ is positive). When this happens, $T^n$ converges uniformly to $P_M$ at a geometric rate, in the sense that there exist $C \geq 0 $ and $\alpha \in (0,1)$ such that
\begin{equation*}
\|T^n - P_M\| \leq C\alpha^n, \quad n\geq 1.
\end{equation*}
It turns out that $c(M_1,M_2)=1$ can only happen in an infinite-dimensional space. For $c(M_1,M_2)=1$, we do not have uniform convergence, but we still have strong convergence (for every $x \in H$, $\|T^nx - P_Mx\| \to 0$) by Theorem \ref{von neumann} (von Neumann). In 2009, Bauschke, Deutsch and Hundal \cite{BaDeHu09} proved that in this case, convergence is arbitrarily slow, in the sense that for any monotonically decreasing sequence $(\lambda_n)$ in $[0,1]$ tending to $0$, there exists $x_\lambda \in H$ such that
\begin{equation*}
\|T^n(x_\lambda) - P_M(x_\lambda)\| \geq \lambda_n, \quad n\geq 1.
\end{equation*}
Hence we have a dichotomy:
\begin{equation*}
\begin{aligned}
&c(M_1,M_2) < 1 \implies \textnormal{ convergence at a uniform geometric rate}, \\
&c(M_1,M_2) = 1 \implies \textnormal{ arbitrarily slow convergence}.
\end{aligned}
\end{equation*}
In 2012, Badea, Grivaux and M\"uller \cite{BaGrMu12} extended the notion of Friedrichs angle and the results discussed above to the case of $J\geq2$ closed subspaces $M_1,\dots,M_J$. In particular, the same dichotomy still holds, except with $c(M_1,M_2)$ replaced by $c(M_1,\dots, M_J)$.

The most recent result concerning the rate of convergence is the following. Let $M = \bigcap_{j=1}^J M_j$ be the intersection of $J$ closed subspaces, and $T=P_{M_J}\dots P_{M_1}$. In 2017, Badea and Seifert \cite{BaSe16} proved that there exists a dense subspace $H_0$ of $H$ such that for any $x_0 \in H_0$, we have
\begin{equation*}
\|T^n(x_0) - P_M(x_0)\| = o(n^{-k}), \quad k\geq1.
\end{equation*}
They referred to this as `super-polynomially fast' convergence \cite{BaSe16}. Their result tells us that given $\varepsilon>0$, even in the bad case where $c(M_1,\dots, M_J)=1$, if we pick an initial point where we have slow convergence, we are a distance of less than $\varepsilon$ away from a point where we have super-polynomially fast convergence. 

For applications, it would be useful to be able to get a better idea of where the points (elements of $H$) that give fast and slow convergence are located. However, fairly little is known about this. Nevertheless, there is a conjecture by Deutsch and Hundal as to where points that give slow convergence can be found \cite{DeHu10}. The paper proves equivalent conditions for $c(M_1,\dots,M_J)<1$, from which it follows that
\begin{equation*}
c(M_1,\dots, M_J) = 1 \iff \sum_{j=1}^J M_j^\perp \textnormal{ is not closed in } H.
\end{equation*}
In this case, we know that given a monotonically decreasing sequence $(\lambda_n)$ in $[0,1]$ tending to $0$, there exists $x_\lambda \in H$ such that
\begin{equation*}
\|T^n(x_\lambda) - P_M(x_\lambda)\| \geq \lambda_n, \quad n\geq 1.
\end{equation*}

Deutsch and Hundal's conjecture is that for $(\lambda_n)$ tending to $0$ sufficiently slowly,
\begin{equation*}
x_\lambda \in M^\perp \setminus \sum_{j=1}^J M_j^\perp.
\end{equation*}
This would be useful in knowing how to avoid points where we have slow convergence, but it remains to be seen if this conjecture is true.
\newpage
\subsection{Conclusion}

In this dissertation, we presented proofs of some well known results concerning the method of alternating projections. These include an original proof of von Neumann's theorem \cite{von49}, clarifying a remark in Sakai's paper \cite{Sak95}, and simplifying Amemiya and Ando's proof \cite{AmAn65} for the case of orthogonal projections. The key results are that $(x_n)$ always converges weakly, $(x_n)$ converges in norm when $(j_n)$ is quasiperiodic (and in particular periodic), and that we may find a sequence $(j_n)$, such that for any given vector $x_0 \in H$, we may find three closed subspaces intersecting only at the origin, for which $(x_n)$ does not converge in norm.

Beyond those mentioned in the dissertation, we do not know of any other results regarding the convergence of $(x_n)$. In particular, given a sequence $(j_n)$ that is not quasiperiodic, and with none of the closed subspaces $M_j$ finite-dimensional, no further results are available to determine whether $(x_n)$ converges in norm. Whether in the future we will be able to say more about the convergence of $(x_n)$ remains to be seen.

\newpage
\renewcommand{\abstractname}{Acknowledgements}
\begin{abstract}
\thispagestyle{plain}
I would like to thank David Seifert for sparking my interest in the method of alternating projections, and for taking the time to supervise my dissertation.
\end{abstract}

\newpage
\nocite{*}
\bibliography{master}
\bibliographystyle{IEEEtranS}

\end{document}